\documentclass[reqno,11pt,twoside]{article}
\usepackage[a4paper,hmargin=1in,vmargin=1in]{geometry}
\usepackage{amsmath,amssymb,amsthm}
\usepackage{bm}
\usepackage{graphicx}
\usepackage{etoolbox}
\usepackage{ebgaramond}
\usepackage[T1]{fontenc}
\newtheoremstyle{plainsc}
{\topsep}
{\topsep}
{\itshape}
{}
{\large\scshape}
{}
{5pt plus 1pt minus 1pt}
{\thmname{#1}\thmnumber{ #2}\thmnote{ (#3)}.}

\newtheoremstyle{definitionsc}
{}
{}
{\normalfont}
{}
{\large\scshape}
{ }
{ }
{\thmname{#1}\thmnumber{ #2}\thmnote{ (#3)}.}

\newtheoremstyle{remarksc}
{0.5\topsep}
{0.5\topsep}
{\normalfont}
{}
{\large\itshape}
{}
{ }
{\thmname{#1}\thmnumber{ #2}\thmnote{ (#3)}.}
\numberwithin{equation}{section}
\theoremstyle{plainsc}
\newtheorem{theorem}{Theorem}[section]
\newtheorem{proposition}[theorem]{Proposition}
\newtheorem{corollary}[theorem]{Corollary}

\theoremstyle{definitionsc}
\newtheorem{definition}[theorem]{Definition}
\newtheorem{example}[theorem]{Example}

\theoremstyle{remarksc}
\newtheorem{remark}[theorem]{Remark}

 \newcommand{\addQEDstyle}[2]{\AtBeginEnvironment{#1}{\pushQED{\qed}\renewcommand{\qedsymbol}{#2}}\AtEndEnvironment{#1}{\popQED}}
 \addQEDstyle{example}{$\triangleleft$}
  \addQEDstyle{definition}{$\triangleleft$}
 \addQEDstyle{remark}{$\blacktriangleleft$}
\newcommand{\balpha}{\bm{\alpha}}
\newcommand{\bell}{\bm{\ell}}
\newcommand{\bzeta}{\bm{\zeta}}
\DeclareMathOperator{\changes}{\mathbf{Ch}}

\newcommand{\sign}{\operatorname{sign}}
\newcommand{\const}{\operatorname{const}}

\newcommand{\pmset}[1]{\mathfrak{Z}^{#1}}

\newcommand{\er}{\mathrm{e}}
\newcommand{\ir}{\mathrm{i}}
\newcommand{\dr}{\mathrm{d}}
\renewcommand{\tilde}{\widetilde}

\makeatletter
\renewcommand*\env@matrix[1][\arraystretch]{%
  \edef\arraystretch{#1}%
  \hskip -\arraycolsep
  \let\@ifnextchar\new@ifnextchar
  \array{*\c@MaxMatrixCols c}}
\makeatother
\renewcommand*{\arraystretch}{1.25}
\makeatletter
\renewcommand{\pod}[1]{\allowbreak\mathchoice
  {\if@display \mkern 18mu\else \mkern 8mu\fi (#1)}
  {\if@display \mkern 18mu\else \mkern 8mu\fi (#1)}
  {\mkern4mu(#1)}
  {\mkern4mu(#1)}
}
\makeatother
\renewcommand\footnotemark{}
\usepackage{fancyhdr}
\usepackage[colorlinks,citecolor=blue,linkcolor=blue]{hyperref}
\begin{document}
\title{Inverse Steklov spectral problem for  curvilinear polygons}
\author{
Stanislav Krymski\hspace{-3ex}
\thanks{{\bf SK:} Department of Mathematics and Computer Science, St. Petersburg State University, 14th Line 29B, Vasilyevsky Island, St. Petersburg 199178, Russia; krymskiy.stas@yandex.ru}
\and
Michael Levitin\hspace{-3ex}
\thanks{{\bf ML:} Department of Mathematics and Statistics,
University of Reading, Whiteknights, PO Box 220, Reading RG6 6AX, UK;
M.Levitin@reading.ac.uk; \url{http://www.michaellevitin.net}}
\and Leonid Parnovski\hspace{-3ex}
\thanks{{\bf LP:} Department of Mathematics, University College London,
Gower Street, London WC1E 6BT, UK;
leonid@math.ucl.ac.uk}
\and Iosif Polterovich\hspace{-3ex}
\thanks{{\bf IP:} D\'e\-par\-te\-ment de math\'ematiques et de
sta\-tistique, Univer\-sit\'e de Mont\-r\'{e}al, CP 6128 succ
Centre-Ville, Mont\-r\'{e}al QC  H3C 3J7, Canada;
iossif@dms.umontreal.ca; \url{http://www.dms.umontreal.ca/\~iossif}}
\and David A. Sher
\thanks{{\bf DAS:} Department of Mathematical Sciences, DePaul University, 2320 N Kenmore Ave, Chicago, IL 60614, USA;
dsher@depaul.edu}
}
\date{\small July 31, 2020; final version to appear in IMRN (journal typesetting may differ), doi:10.1093/imrn/rnaa200}
\maketitle
\begin{abstract}
This paper studies the inverse Steklov spectral problem for curvilinear polygons.  For generic curvilinear polygons with angles  less than $\pi$, we  prove  that the asymptotics of Steklov eigenvalues obtained in \cite{LPPS}  determines, in a constructive manner, the number of  vertices and the properly ordered sequence of side lengths, as well as the angles up to a certain equivalence relation.   We also  present counterexamples to this statement  if the  generic assumptions fail.  
In particular, we show that there exist non-isometric triangles with asymptotically  close Steklov spectra. Among other techniques, we use a version of the Hadamard--Weierstrass factorisation theorem, allowing us to reconstruct a trigonometric  function from the asymptotics of its roots.
\end{abstract}
\tableofcontents
\section{Introduction and main results}
\subsection{Direct and inverse Steklov spectral problems}
Let $\Omega\subset\mathbb{R}^2$ be a bounded connected planar domain with connected Lipschitz  boundary $\partial\Omega$ of length  
$L=|\partial\Omega|$.
Consider the Steklov eigenvalue problem
\begin{equation}\label{eq:Steklov}
\Delta u =0\quad \text{in }\Omega,\qquad\qquad \frac{\partial u}{\partial n}=\lambda u\quad \text{on }\partial \Omega,
\end{equation}
with $\lambda$ being the spectral parameter, and $\dfrac{\partial u}{\partial n}$ being the exterior normal derivative. 

\smallskip

The spectrum of the Steklov problem is discrete: 
\[
0=\lambda_1(\Omega)<\lambda_2(\Omega)\le\dots \le \lambda_m(\Omega) \le \dots \nearrow +\infty.
\]
Equivalently, $\lambda_m$ may be viewed as the eigenvalues of the \emph{Dirichlet-to-Neumann map} $\mathcal{D}_\Omega$:
\[
\mathcal{D}_{\Omega}:H^{1/2}(\partial\Omega)\to H^{-1/2}(\partial\Omega),\qquad \mathcal{D}_{\Omega} f:=\left.\frac{\partial \mathcal{H}_\Omega f}{\partial n}\right|_{\Omega},
\] 
where  $\mathcal{H}_\Omega f$ denotes the harmonic extension of $f$ to $\Omega$.  

If the boundary $\partial \Omega$ is piecewise $C^1$, the Steklov eigenvalues have the following Weyl-type asymptotics (see \cite{Ag2006}):
\begin{equation}\label{eq:agranovich}
\lambda_m=\frac{\pi m}{ |\partial\Omega|} +o(m)\qquad\text{as }m\to+\infty.
\end{equation}

In the past decade, there has been a lot of research on the Steklov eigenvalue problem, see \cite{GP, LPPS} and references therein.
In particular, a significant amount of  information has been obtained on the {\it direct} spectral problem,  which is concerned with the dependence of the Steklov eigenvalues on the  underlying geometry.  The present paper  focuses on  the {\it inverse} spectral problem: which geometric properties of $\Omega$ are determined by the Steklov spectrum? 

Let
\begin{equation}\label{eq:Lambda}
\Lambda=\Lambda_{\Omega}:=\{\lambda_1, \lambda_2,\dots\}
\end{equation}
be a multiset  given by the Steklov eigenvalues of $\Omega$ with the account of multiplicities. We say that two domains $\Omega_1$ and $\Omega_2$ are {\it Steklov isospectral} if 
$\Lambda_{\Omega_1} = \Lambda_{\Omega_2}$. Interestingly enough,  no examples of non-isometric Steklov isospectral planar domains are presently known \cite[Open problem 6] {GP}; we refer also to \cite{Edw93b,MS, JS, JS18}  for some related results and conjectures.  At the same time, Steklov spectral invariants of planar domains are also quite scarce.
It follows from Weyl's law \eqref{eq:agranovich} that the perimeter of $\Omega$ is such an invariant. Moreover, if the boundary of $\Omega$ is smooth, the Steklov spectrum determines the number of boundary components and their lengths \cite{GPPS}. However, for smooth simply connected planar domains, extracting further  geometric information from the Steklov problem is quite difficult. In part, the reason is that in this case the Dirichlet-to-Neumann map $\mathcal{D}_{\Omega}$ is a pseudodifferential operator of order one on the circle, and the remainder estimate in Weyl's law \eqref{eq:agranovich} could be significantly improved  \cite{Ros, Ed}:
\begin{equation}
\label{RGM}
\lambda_{2m}=\lambda_{2m+1}+O\left(m^{-\infty}\right)=\frac{2\pi m}{|\partial\Omega|}+O\left(m^{-\infty}\right), \qquad m\to +\infty.
\end{equation}
As a result, no other spectral invariants except the perimeter could be obtained from the eigenvalue asymptotics on the polynomial scale.
\begin{definition}
\label{def:quasiisosp}
We say that two bounded planar domains $\Omega_1$ and $\Omega_2$ are \emph{Steklov quasi-isospectral} if their respective Steklov eigenvalues are asymptotically $o(1)$-close:
$\lambda_m(\Omega_1)-\lambda_m(\Omega_2)=o(1)$ as $m\to \infty$.
\end{definition}

In particular, any two Steklov isospectral domains are also Steklov quasi-isospectral. It also follows from \eqref{RGM} that {\em all} smooth simply-connected planar domains of given perimeter are Steklov quasi-isospectral; moreover, in view of \eqref{RGM},  $o(1)$-closeness of the corresponding eigenvalues  immediately implies $o\left(m^{-\infty}\right)$-closeness as $m\to \infty$, cf. Remark \ref{rem:closeness}.

\begin{remark} Similarly to \cite{KurSuhr}, one may also call two Steklov quasi-isospectral planar domains \emph{asymptotically Steklov isospectral}. Our choice of terminology is motivated by Corollary \ref{cor:quasi} below. 
\end{remark}

In the present paper we investigate the inverse spectral problem on curvilinear polygons. In this case, the asymptotic formula \eqref{RGM} does not hold even with a $o(1)$ error term.
 In fact, as was recently shown in \cite{LPPS}, the eigenvalue asymptotics depends in a delicate way on the number of vertices, the side lengths and the angles at the corner points. Moreover, \cite[Corollary 1.6]{LPPS}  implies that curvilinear polygons having the same respective edge lengths and angles are quasi-isospectral, see also Theorem \ref{thm:maindirect}.
It is therefore natural to ask whether these geometric features of a curvilinear polygon are determined by the Steklov spectrum. 
\subsection{Steklov spectrum of a curvilinear polygon}
Let $\mathcal{P}=\mathcal{P}(\balpha,\bell)$ be a curvilinear polygon with angles $\balpha=(\alpha_1,\dots,\alpha_n)$ and side lengths $\bell=(\ell_1,\dots,\ell_n)\in\mathbb{R}^n_+$
(see Figure \ref{fig:fig1}).  Note that the vertices $V_j$ and the edges $I_j$ (of length $\ell_j$) are enumerated clock-wise. The angle $\alpha_j$ at the vertex $V_j$ is formed by the edges $I_j$ and $I_{j+1}$, $j=1,\dots, n$. Here and further on we use cyclic subscript identification $n+1 \equiv 1$.
Throughout the paper, we assume that the sides of the polygon are smooth, and  that $\balpha=(\alpha_1,\dots,\alpha_n)\in(0,\pi)^n$, i.e.\ we only consider polygons with angles less than $\pi$, see also Remark \ref{lessthanpi}.
We  denote by
\begin{equation}\label{eq:Ldef}
L=L_{\bell}=|\partial\mathcal{P}|:=\ell_1+\dots\ell_n
\end{equation}
the perimeter of $\mathcal{P}$.
\begin{figure}[hbt]
\begin{center}
\includegraphics[width=4in]{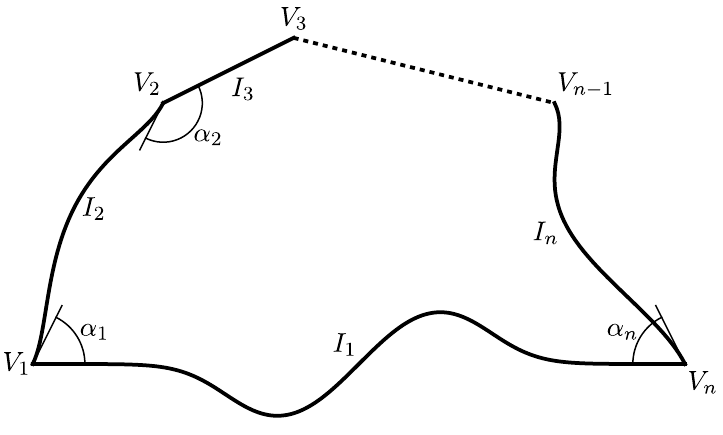}
\caption{A curvilinear polygon\label{fig:fig1}}
\end{center}
\end{figure}

An asymptotic characterisation of the Steklov spectrum of $\mathcal{P}(\balpha, \bell)$, denoted as in \eqref{eq:Lambda}  by $\Lambda_{\mathcal{P}(\balpha, \bell)}$, 
in terms of the zeros of a certain trigonometric polynomial determined by $\balpha, \bell$, was obtained in \cite{LPPS}. We recall this construction below.

For given vectors $\balpha\in(0,\pi)^n$, $\bell\in\mathbb{R}^n_+$, define the \emph{characteristic polynomial} of the Steklov problem \eqref{eq:Steklov} on $\mathcal{P}(\balpha, \bell)$ as the trigonometric polynomial\footnote{Note a  slight change of notation compared to \cite{LPPS}. Note also that in this paper we use the term ``trigonometric polynomial'' in a generalised sense as we allow the frequencies to be incommensurable.} of a real variable 
$\sigma$,
\begin{equation}\label{eq:Fdef}
F_{\balpha,\bell}(\sigma):=\sum_{\bzeta\in\pmset{n}_+} \mathfrak{p}_{\bzeta} \cos(|\bell\cdot\bzeta|\sigma)-\prod_{j=1}^n\sin\left(\frac{\pi^2}{2\alpha_j}\right),
\end{equation}
where
\[
\pmset{n}=\{\pm1\}^n,\qquad \pmset{n}_+=\{1\}\times\pmset{n-1}\subset\pmset{n}
\]
(that is, $\pmset{n}$ is the set of all sequences of $\pm 1$ of length $n$, and $\pmset{n}_+$ is its subset which includes only sequences starting with $+1$),
\begin{equation}
\label{eq:pfrakdefn}
\mathfrak{p}_{\bzeta}=\mathfrak{p}_{\bzeta}(\balpha):=\prod_{j\in\changes(\bzeta)}\cos\frac{\pi^2}{2\alpha_j},
\end{equation}
and for a vector $\bzeta=(\zeta_1,\dots,\zeta_n)\in \pmset{n}$ with cyclic identification $\zeta_{n+1}\equiv\zeta_1$,
\begin{equation}\label{eq:Cintrodef}
\changes(\bzeta):=\{j \in\{1,\dots,n\}\mid \zeta_j\ne \zeta_{j+1} \} .
\end{equation}
We  denote the class of all possible characteristic polynomials by
\[
\label{eq:charpol}
\mathcal{F}:=\left\{F_{\balpha,\bell}(\sigma)\mid \balpha\in(0,\pi)^n, \bell\in\mathbb{R}^n_+, n\in\mathbb{N}\right\}.
\]
If $F_{\balpha,\bell}(0)=0$, let
\[
2m_0=2m_{0,\balpha,\bell}
\]
be the multiplicity of zero as a root of $F_{\balpha,\bell}$ (this mulitplicity is always even since  $F_{\balpha,\bell}(\sigma)$ is an even function of $\sigma$), otherwise set $m_0=0$.

Denote by $\sigma_1\le\sigma_2\le\dots$, the non-negative roots of \eqref{eq:Fdef} taken with account of their  multiplicities (except $\sigma=0$ which, if present, is taken with \emph{half} its  multiplicity, that is $m_0$). We call them  \emph{quasi-eigenvalues} of the Steklov problem \eqref{eq:Steklov} on $\mathcal{P}(\balpha, \bell)$. Let
\begin{equation}\label{eq:Sigmadef}
\Sigma=\Sigma_{\balpha,\bell}:=\{\sigma_1, \sigma_2, \dots\},
\end{equation}
with account of multiplicities as above. 
\begin{remark}
\label{quantum0}
As was shown in \cite[Subsection 2.5]{LPPS}, the quasi-eigenvalues  $\Sigma$ may be also viewed as the square roots of the eigenvalues of a certain quantum graph Laplacian $\mathcal{G}_{\balpha,\bell}$, where the metric graph is circular and  is modelled  on the boundary of $\mathcal{P}$, and the matching conditions are determined by the angles at the vertices. See  Remark \ref{rem:quantum} for a further discussion. 
\end{remark}
One of the main results of \cite{LPPS} is
\begin{theorem}[{\cite[Theorem 2.16]{LPPS}}]\label{thm:maindirect} Let  $\mathcal{P}(\balpha,\bell)$ be a curvilinear polygon with angles $\balpha=(\alpha_1,\dots,\alpha_n)\in(0,\pi)^n$, and side lengths $\bell=(\ell_1,\dots,\ell_n)\in\mathbb{R}^n_+$, and let $\Lambda_{\mathcal{P}(\balpha,\bell)}$ and $\Sigma_{\balpha,\bell}$ be defined by \eqref{eq:Lambda} and \eqref{eq:Sigmadef}. Then, with some $\epsilon>0$,
\begin{equation}\label{eq:quasibound}
\lambda_m-\sigma_m=O\left(m^{-\epsilon}\right),\qquad\text{as }m\nearrow +\infty.
\end{equation}
\end{theorem}
Consequently, as was mentioned in the previous subsection, any two curvilinear polygons sharing the vectors $\bell$ and $\balpha$ are Steklov quasi-isospectral.
\begin{remark}
\label{lessthanpi} As was mentioned in \cite{LPPS}, numerical experiments indicate that Theorem \ref{thm:maindirect} holds also for polygons having angles greater or equal than $\pi$,
however there is a technicality in  the proof (which goes back to the methods of \cite{LPPS17}) that requires us to assume that all the angles are less than $\pi$. Still, we believe that all the results of the next section remain valid without this assumption. 
\end{remark}
\subsection{Main results} 
In order to state our main results, we need to introduce some additional notation. Recall that in \cite{LPPS} we distinguish the set of \emph{special angles}
\[
\mathcal{S}:=\left\{\frac{\pi}{2k+1}\mid k\in\mathbb{N}\right\},
\]
and the set of \emph{exceptional angles}
\[
\mathcal{E}:=\left\{\frac{\pi}{2k}\mid k\in\mathbb{N}\right\}.
\]
In what follows we say that a curvilinear polygon is {\it non-exceptional} if it has no exceptional angles.
Additionally, for $\balpha\in(0,\pi)^n$, we will define the corresponding {\em cosine vector}
\begin{equation}\label{eq:calpha}
\mathbf{c}=\mathbf{c}_{\balpha}=\left(c_1,\dots,c_n\right)\in[-1,1]^n,
\end{equation}
where
\[
c_j:=c(\alpha_j):=\cos\frac{\pi^2}{2\alpha_j},\qquad j=1,\dots,n.
\]
Note that an angle $\alpha$ is not special iff  $c(\alpha)\ne 0$, and  $\alpha$ is not exceptional iff $|c(\alpha)|<1$. For an exceptional angle $\alpha=\frac{\pi}{2k}$ with $k\in\mathbb{N}$ we have
\[
\mathcal{O}(\alpha):=c(\alpha)=(-1)^k,
\]
and as in  \cite{LPPS} we will call this quantity the \emph{parity} of an exceptional angle $\alpha$.

\begin{definition}
\label{def:loosely}
We say that  two curvilinear polygons $\mathcal{P}(\balpha,\bell)$ and $\tilde{\mathcal{P}}(\tilde{\balpha}, \tilde{\bell})$ are {\it loosely equivalent} if one can choose the orientation and  the enumeration of vertices of these polygons in such a way that $\bell=\tilde{\bell}$ and either $\mathbf{c}_{\balpha}=\mathbf{c}_{\tilde{\balpha}}$ or $\mathbf{c}_{\balpha}=-\mathbf{c}_{\tilde{\balpha}}$.
\end{definition}
\begin{remark}
\label{ftn:c}
The correct enumeration of the components of $\mathbf{c}_{\balpha}$ depends on the enumeration of the components of $\bell$, since an angle $\alpha_j$ lies between sides $\ell_j$ and $\ell_{j+1}$. For example,  let  $\bell=(\ell_1,\dots,\ell_n)$  and the corresponding cosine vector $\mathbf{c}=(c_1,\dots,c_n)$;   if  the orientation of $\bell$  is changed, say as $(\ell_1,\ell_n,\dots,\ell_2)$, the corresponding cosine vector is given by $(c_n, c_{n-1}, \dots, c_1)$ --- note a shift in the indexing.
\end{remark}
\begin{definition}
\label{def:healthy}
We say that a curvilinear $n$-gon  $\mathcal{P}(\balpha, \bell)$
 is {\it admissible} if 
\begin{equation}\label{eq:ellcond}
\text{the lengths $\ell_1,\dots,\ell_n$ are  incommensurable over $\{-1,0,+1\}$,}
\end{equation}
{\rm(}that is, only a trivial linear combination of  $\ell_1,\dots,\ell_n$ with these coefficients vanishes{\rm)}, and
\begin{equation}\label{eq:alphacond}
\text{all angles $\alpha_1,\dots,\alpha_n$  are not special}.
\end{equation} 
\end{definition}
Clearly, being admissible  is a generic condition for curvilinear polygons. It is essential for our statements to hold, see Remark \ref{rem:adm}.

Let us now formulate the first main result of the paper. It may be thought of as a converse statement to \cite[Corollary 1.6]{LPPS}.
\begin{theorem}
\label{thm:mainnoexc}  
Let $\mathcal{P}$ and $\widetilde{\mathcal{P}}$  be two Steklov quasi-isospectral  admissible  curvilinear polygons.  Suppose that $\mathcal{P}$ is non-exceptional. Then $\widetilde{\mathcal{P}}$ is  loosely equivalent to $\mathcal{P}$.  
\end{theorem}

In order to state the analogue of Theorem \ref{thm:mainnoexc} for polygons with exceptional angles we need additional terminology and notation from \cite{LPPS}. If there are $K>0$ exceptional angles, they split the boundary $\partial\mathcal{P}$ into $K$ \emph{exceptional boundary components} $\mathcal{Y}_1,\dots,\mathcal{Y}_K$; let $n_\kappa$ denote the number of boundary arcs in $\mathcal{Y}_\kappa$, $\kappa=1,\dots,K$. Without loss  of generality we can assume that the vertices of $\mathcal{P}$ are enumerated in such a way that  the endpoint of $\mathcal{Y}_K$ is the vertex $V_n$.  Each exceptional boundary component $\mathcal{Y}_\kappa$ is described by a vector of $n_\kappa$ lengths of its boundary arcs $\bell(\mathcal{Y}_\kappa):=\bell^{(\kappa)}=\left(\ell_1^{(\kappa)},\dots, \ell_{n_\kappa}^{(\kappa)}\right)$, and by a vector of $n_\kappa-1$ non-exceptional angles  between these arcs\footnote{Another slight change of notation compared to \cite{LPPS}.}, $\balpha(\mathcal{Y}_\kappa):=\balpha^{(\kappa)}=\left(\alpha_1^{(\kappa)},\dots, \alpha_{n_\kappa-1}^{(\kappa)}\right)$. Set in this case $\mathbf{c}(\mathcal{Y}_\kappa):=\mathbf{c}^{(\kappa)}=\left(\cos\frac{\pi^2}{2\alpha_1^{(\kappa)}},\dots, \cos\frac{\pi^2}{2\alpha_{n_\kappa-1}^{(\kappa)}}\right)$.  Denote also by $-\mathcal{Y}_\kappa$ the {\it inverse} of  
$\mathcal{Y}_\kappa$ obtained by reversing the  orientation (i.e. reversing the order of the vertices of $\mathcal{Y}_\kappa$).  We call an exceptional boundary component \emph{even} if the parities of exceptional angles at its ends coincide, and \emph{odd} if they are different.

As shown in \cite[Theorem 2.17(b)]{LPPS}, the exceptional  boundary components of a curvilinear polygon contribute to its set of quasi-eigenvalues independently.

Using the notation above, let us introduce the following analogue of Definition \ref{def:loosely} for exceptional boundary components.
\begin{definition}
\label{def:looselycomp}
Let $\mathcal{Y}$ and $\tilde{\mathcal{Y}}$ be two exceptional boundary components. We say that $\mathcal{Y}$ and $\tilde{\mathcal{Y}}$ are {\it loosely equivalent} if  they have the same parity and, by choosing $\tilde{\mathcal{W}}$ to be 
$\tilde{\mathcal{Y}}$ or $-\tilde{\mathcal{Y}}$, we have
$\bell(\mathcal{Y})=\bell(\tilde{\mathcal{W}})$ and either $\mathbf{c}(\mathcal{Y})=\mathbf{c}(\tilde{\mathcal{W}})$ or $\mathbf{c}(\mathcal{Y})=-\mathbf{c}(\tilde{\mathcal{W}})$.
\end{definition} 
In other words, two exceptional boundary components are loosely equivalent if their length vectors coincide modulo, possibly, a reversal of orientation, and once orientation is fixed, their cosine vectors coincide modulo, possibly, a global change of sign. 
\begin{theorem} 
\label{thm:mainexc}  
Let $\mathcal{P}$ and $\widetilde{\mathcal{P}}$ be two Steklov quasi-isospectral  admissible curvilinear polygons.  Suppose that $\mathcal{P}$ has  $K \ge 1$ exceptional boundary components $\mathcal{Y}_\kappa$,
$\kappa=1,\dots, K$.   Then $\widetilde{\mathcal{P}}$ also has $K$ exceptional boundary components 
which could be re-ordered in such a way that,  for any $\kappa$, its $\kappa$-th component becomes loosely equivalent to  $ \mathcal{Y_\kappa}$.
\end{theorem}
\begin{remark} Theorems \ref{thm:mainnoexc} and \ref{thm:mainexc} imply that the number of vertices and the number of exceptional boundary components of a curvilinear polygon are 
Steklov spectral invariants, see also 
Theorem \ref{thm:B}(b).  Note that in Theorem \ref{thm:mainexc}, we cannot obtain any information on how the exceptional boundary components of $\widetilde{\mathcal{P}}$ 
are joined together.
\end{remark}

Theorems \ref{thm:mainnoexc} and \ref{thm:mainexc} follow directly from Theorems \ref{thm:A0} and \ref{thm:B}  below.
\begin{theorem}
\label{thm:A0}
Two curvilinear polygons are Steklov quasi-isospectral if and only if  their characteristic polynomials coincide.
\end{theorem}
The ``if''  direction immediately follows from \eqref{eq:quasibound}. The ``only if'' part  is essentially proved in two steps. First, in Theorem \ref{prop:Aprime} we show that the polynomial $F_{\balpha,\bell}(\sigma)$   is uniquely determined by the collection of  its nonnegative zeros $\Sigma_{\balpha,\bell}$, which are the quasi-eigenvalues of $\mathcal{P}$.  This easily follows from the well-known Hadamard--Weierstrass factorisation theorem for entire functions \cite{Con}. Second, we deduce  from a general property of the zeros of almost periodic functions \cite[Theorem 6]{KurSuhr} and the asymptotic formula \eqref{eq:quasibound}  that the collection of quasi-eigenvalues $\Sigma_{\balpha, \bell}$ coincides for all Steklov quasi-isospectral curvilinear polygons.

Theorems \ref{thm:A0} and \ref{prop:Aprime} immediately imply
\begin{corollary}\label{cor:quasi} Two curvilinear polygons are Steklov quasi-isospectral if and only if their quasi-eigenvalues coincide.
\end{corollary}

\begin{remark} 
\label{rem:closeness}
Theorem \ref{thm:A0} together with formula  \eqref{eq:quasibound} also imply that if two curvilinear polygons $\mathcal{P}$ and $\widetilde{\mathcal{P}}$ are Steklov quasi-isospectral, there exists an $\epsilon>0$ such that $\lambda_m(\mathcal{P}) -\lambda_m(\widetilde{\mathcal{P}}) =O\left(m^{-\epsilon}\right)$ as $m \to \infty$. 
\end{remark}

We also prove the following  constructive modification of Theorem \ref{thm:A0}:
\begin{theorem}  
 \label{thm:A}
The characteristic polynomial   $F_{\balpha,\bell}(\sigma)$  defined by \eqref{eq:Fdef} can be reconstructed algorithmically from the Steklov spectrum of a corresponding  curvilinear polygon  
$\mathcal{P}(\balpha,\bell)$.
\end{theorem}
The proof of Theorem \ref{thm:A}  also uses the Hadamard--Weierstrass factorisation, but does not rely on the results of \cite{KurSuhr}.
 Instead, we use in an essential way the polynomial decay of the error estimate  \eqref{eq:quasibound},   see subsection \ref{subsec:AA} for details.
\begin{theorem}
\label{thm:B}
Given the  characteristic polynomial $F(\sigma)=F_{\balpha,\bell}(\sigma)$ of an admissible  curvilinear polygon $\mathcal{P}(\balpha,\bell)$, we can recover the number of vertices $n$ and 
the number of exceptional angles $K\ge 0$.
Moreover, 
\begin{enumerate}
\item[\rm{(a)}] If there are no exceptional angles, that is $K=0$, we can recover the vector of side-lengths $\bell$ modulo cyclic shifts and a reversal of orientation, and, once the enumeration of vertices is fixed (cf. Remark \ref{ftn:c}), we can also recover the vector $\mathbf{c}_{\balpha}$ modulo a global change of sign.
\item[\rm{(b)}] If the number of exceptional angles is $K\ge 1$, then for each exceptional boundary component $\mathcal{Y}_\kappa$, $\kappa=1,\dots,K$, we can determine whether it is even or odd, and obtain the number $n_\kappa$ of its constituent boundary arcs, the vector of their lengths $\bell^{(\kappa)}$ modulo a reversal of orientation,  and, once the orientation is fixed, we can also recover the vector $\mathbf{c}^{(\kappa)}$ modulo a global change of sign.
\end{enumerate}
\end{theorem}
It immediately  follows from Theorems \ref{thm:A} and \ref{thm:B} that all the geometric data in Theorem \ref{thm:B} can be reconstructed from the Steklov spectrum of an admissible  curvilinear polygon.

The proof of Theorem \ref{thm:B} is fully constructive, in a sense that all the operations required to extract the geometric data from the characteristic polynomial may be easily done ``by hand'' or implemented using symbolic computations, see subsection \ref{subsec:proofthmb}. By contrast, a numerical implementation of the algorithm of 
Theorem \ref{thm:A} may not be straightforward.

\begin{remark}
\label{rem:adm}
If either of the admissibility conditions  \eqref{eq:ellcond} or \eqref{eq:alphacond} is not satisfied, we can construct a number of examples of different curvilinear polygons with the same characteristic polynomial,
see  subsection \ref{subsec:examples}. The admissibility assumption is therefore necessary for the validity of Theorem \ref{thm:B}. Hence, in view of Theorem \ref{thm:A0}, Theorems \ref{thm:mainnoexc} and \ref{thm:mainexc} require this assumption as well.
\end{remark}
\section{Proofs of Theorems \ref{thm:A0} and  \ref{thm:A}}
\subsection{Some auxiliary facts}
We will use  the following results proved in \cite{LPPS}.
\begin{proposition}\label{prop:summary} Let $\mathcal{P}=\mathcal{P}(\balpha,\bell)$ be a curvilinear polygon, with $\balpha\in(0,\pi)^n$, $\bell\in\mathbb{R}_+^n$, and $L$ given by \eqref{eq:Ldef}. With the sequences $\Sigma_{\balpha,\bell}$ and $\Lambda_{\mathcal{P}}$ defined as above, we have
\begin{enumerate}
\item[{\normalfont(a)}] As $\sigma\to+\infty$,
\[
\#(\Sigma_{\balpha,\bell}\cap [0, \sigma))= \#(\Lambda_{\mathcal{P}}\cap [0, \sigma))+O(1)=\frac{L}{\pi} \sigma+O(1).
\]
\item[{\normalfont(b)}] There exists a constant $N=N_{\balpha,\bell}\in\mathbb{N}$ such that  for every interval $I\subset\mathbb{R}_+$ of length one
\[
\#(\Sigma_{\balpha,\bell}\cap I)\le N\qquad\text{and}\qquad\#(\Lambda_{\mathcal{P}}\cap I)\le N.
\]
\end{enumerate}
\end{proposition}
Part (a) is just a one-term Weyl's asymptotic formula for the Steklov quasi-eigenvalues and eigenvalues, respectively,  which are taken from  \cite[formula (2.31)]{LPPS}  and \cite[Proposition 2.30]{LPPS}  (the former in turn follows from  \cite[Lemma 3.7.4]{BeKu13} and the quantum graph analogy \cite[Theorem 2.24]{LPPS}).  Part (b)  immediately follows from part (a).  

The polynomial \eqref{eq:Fdef} can be equivalently re-written as
\begin{equation}\label{eq:Fdef1}
F_{\balpha,\bell}(\sigma)=\sum_{\bzeta\in\pmset{n}} \frac{1}{2}\mathfrak{p}_{\bzeta}\cos(|\bell\cdot\bzeta|\sigma)-\prod_{j=1}^n\sin\frac{\pi^2}{2\alpha_j}.
\end{equation}
We note that
\[
\mathfrak{p}_{\bzeta}=\mathfrak{p}_{-\bzeta}.
\]
We additionally set
\begin{equation}\label{eq:Tdef}
\mathcal{T}=\mathcal{T}_n=\mathcal{T}_{n,\bell}:=\{|\bell\cdot\bzeta|: \bzeta\in\pmset{n}_+\}.
\end{equation}
Then the set $\mathcal{T}$ has at most $2^{n-1}$ distinct elements $t_1,\dots,t_{\# \mathcal{T}}$ and \eqref{eq:Fdef}  can be also re-written as
\begin{equation}\label{eq:Fdef2}
F_{\balpha,\bell}(\sigma)=\sum_{k=1}^{\# \mathcal{T}}  r_k \cos(t_k \sigma) - r_0,
\end{equation}
where the coefficients $r_k$, $k=0,1,\dots,\# \mathcal{T}$, depend non-trivially on $\balpha$; if for some $k\ge 1$ we have $t_k=0\in\mathcal{T}$, then the corresponding term $r_k$ is incorporated in $r_0$.
We can also further re-write  \eqref{eq:Fdef2} as
\begin{equation}\label{eq:Fdef3}
F_{\balpha,\bell}(\sigma)=\sum_{k=1}^{\# \mathcal{T}}  \frac{r_k}{2} \left(\er^{-i t_k \sigma}+\er^{i t_k \sigma}\right) - r_0.
\end{equation}

\subsection{Infinite product formula} 
\label{subsec:infprod}
Our first objective is to prove the following result.
\begin{theorem}
\label{prop:Aprime}
Given the collection of quasi-eigenvalues $\Sigma_{\balpha,\bell}=\{\sigma_m\}$ of a curvilinear polygon $\mathcal{P}(\balpha,\bell)$, we can recover the corresponding characteristic polynomial $F_{\balpha,\bell}$ uniquely.
\end{theorem}
We start with the following easy corollary of the Hadamard--Weierstrass factorisation Theorem. We recall that for an entire function $f:\mathbb{C}\to\mathbb{C}$, its \emph{order} $\rho$ 
is defined as 
\[
\rho:=\inf\left\{r\in\mathbb{R}: f(z)=O\left(\er^{|z|^r}\right)\text{ as }|z|\to\infty\right\}.
\]

\begin{theorem}\label{thm:infiniteproduct1}  Let $f:\mathbb{C}\to\mathbb{C}$ be an even entire function of order one with a  zero of order $2m_0$ at $z=0$, and non-zero zeros $\pm \gamma_j$  repeated with multiplicities; denote by $\Gamma$ the sequence (with multiplicities) consisting of $m_0$ zeros and $\gamma_j$. Then there exists a constant $C$ such that
\[
f(z)=CQ_{\Gamma}(z),
\]
where
\[
Q_{\Gamma}(z):=z^{2m_0}\prod_{\gamma_j\in\Gamma\setminus\{0\}}\left(1-\frac{z^2}{\gamma_j^2}\right),
\]
\end{theorem}

\begin{proof} By the Hadamard--Weierstrass factorisation Theorem \cite{Con} applied to $f$, we obtain
\[
f(z)=z^{2m_0}\er^{g(z)}  \prod_{\gamma_j\in\Gamma\setminus\{0\}} E_1\left(\frac{z}{\gamma_j}\right)E_1\left(-\frac{z}{\gamma_j}\right),
\]
where $g(z)$ is a polynomial of degree less than or equal to one, and the \emph{primary} (or \emph{elementary}) \emph{factors} $E_1(w)$ are defined by 
\[
E_1(w):=(1-w)\er^w.
\]
We note that $E_1(w)E_1(-w)=(1-w^2)$, and that, since $f(z)$ is even, so should be $g(z)$. As $g(z)$ is also linear, it is therefore a constant. The result follows immediately.
\end{proof}

Let now $\balpha,\bell$ be arbitrary, and let $F_{\balpha,\bell}\in\mathcal{F}$. Theorem \ref{thm:infiniteproduct1} immediately implies
\begin{theorem}\label{thm:infiniteproduct} There exists a constant $C=C_{\balpha,\bell}$ such that $F_{\balpha,\bell}(\sigma)=CQ_{\Sigma}(\sigma)$, where
\begin{equation}\label{eq:defofq}
Q_{\Sigma}(\sigma):=\sigma^{2m_0}\prod_{\sigma_j\in\Sigma_{\balpha,\bell}\setminus\{0\}}\left(1-\frac{\sigma^2}{\sigma_j^2}\right),
\end{equation}
with $m_0$ being the multiplicity of zero in $\Sigma_{\balpha,\bell}$.
\end{theorem}
\subsection{Recovering a trigonometric polynomial from an infinite product}\label{subs:recovery}
Consider the \emph{mean}  operator $\bf{M}$ defined on the space of almost periodic functions on $\mathbb{R}$ by the formula
\[
{\bf M}[f]:=\lim_{t\to\infty}\frac 1t\int_0^t f(s)\, \dr s,
\]
and consider additionally the function
\[
(\mathcal{A}[f])(z):={\bf M}\left[\er^{-\ir s z}f(s)\right]
\]
whose support determines the set of frequencies of $f$ \cite{Be}.
Note that we are not dealing with any continuity or boundedness of $\mathcal{A}$ so we do not need to specify a norm. It is, however, evident that  $\mathcal{A}[f]$ is linear in $f$. Furthermore, for a constant $q\in\mathbb{R}$, by direct computation,
\[
\mathcal{A}[\er^{\ir q s}](z)=\lim_{t\to\infty}\frac{1}{t}\int_0^t \er^{\ir s (q-z)}\, \dr s=\begin{cases} 0, &\text{if } z\neq q;\\ 1, &\text{if }  z=q.\end{cases}
\]
Also, by an easy argument,
\begin{equation}\label{eq:Aofosmall}
\mathcal{A}[f](z)=0\qquad\text{whenever a function $f$ is $o(1)$ at $+\infty$}.
\end{equation}

But now recall from \eqref{eq:Fdef3} and  Theorem \ref{thm:infiniteproduct}  that
\[
F_{\balpha,\bell}(\sigma)=\sum_{k=1}^{\# \mathcal{T}}  \frac{r_k}{2} \left(\er^{-i t_k \sigma}+\er^{i t_k \sigma}\right) - r_0=CQ_{\Sigma}(\sigma)
\]
with some constant $C$. Thus, the set of frequencies $\mathcal{T}$ of $F$ can be recovered via
\[
\mathcal{T}=\left\{z\ge 0: \mathcal{A}[Q](z)\ne 0\right\}.
\]
The coefficients $r_j$ are then recovered via
\[
r_j=2C\mathcal{A}[Q](t_j)\quad\text{for }0\ne t_j\in\mathcal{T};\qquad r_0=-C\mathcal{A}[Q](0),
\]
and the unknown constant $C$ can be found from the condition that the coefficient $r_k$ corresponding to the maximal element of $\mathcal{T}$ should be equal to one:
\[
C=\frac{1}{2 \mathcal{A}[Q](\max\mathcal{T})}.
\]
This proves Theorem  \ref{prop:Aprime}.

\begin{proof}[Proof of Theorem \ref{thm:A0}] As mentioned in the Introduction, the ``if" part follows directly from \eqref{eq:quasibound}, and it remains to prove the ``only if'' part. Consider two Steklov quasi-isospectral curvilinear polygons $\mathcal{P}$ and $\tilde{\mathcal{P}}$, with $\Sigma$ and $\tilde{\Sigma}$ being their corresponding sets of quasi-eigenvalues. By Theorem \ref{thm:maindirect}, these sets of quasi-eigenvalues differ by $o(1)$ at infinity. Moreover, each set of quasi-eigenvalues is a set of 
zeros of some characteristic polynomial of the form \eqref{eq:Fdef} which is an almost periodic function with all real roots. Therefore, by \cite[Theorem 6]{KurSuhr}, which implies that in this case  two almost periodic functions with asymptotically close zeros have exactly the same zeros, we  have $\Sigma=\tilde{\Sigma}$. An application of Theorem \ref{prop:Aprime}  completes  the proof. 
\end{proof}
\subsection{Another infinite product}
\label{subsec:AA}
In this section we have a sequence  $\Lambda=\{\lambda_m\}$ for which $\lambda_m=\sigma_m+O(m^{-\epsilon})$ for some (unknown) sequence $\Sigma=\{\sigma_m\}$ of roots of an unknown trigonometric polynomial $F\in\mathcal{F}$ with some unknown $\epsilon>0$. We will explain how to recover $F(\sigma)$ from this information.

The key idea of this proof is that, motivated by \eqref{eq:defofq}, we may define a similar ``infinite product'' with $\sigma_m$ replaced by $\lambda_m$. Suppose that $n_0$ elements of $\Lambda$ are equal to zero. (In fact, in the inverse Steklov problem, since $\Lambda$ is the sequence of actual eigenvalues, we always have $n_0=1$.) Then set
\begin{equation}\label{eq:defofr}
Q_\Lambda(\sigma):=\sigma^{2n_0}\prod_{m=n_0+1}^{\infty}\left(1-\frac{\sigma^2}{\lambda_m^2}\right).
\end{equation}

Consider the following ratio,  which we for the moment  compute formally, after some simplifications, as
\[
\frac{Q_{\Lambda}(\sigma)}{Q_{\Sigma}(\sigma)}
=\frac{\sigma^{2n_0}\prod\limits_{m=n_0+1}^{\infty}\left(1-\frac{\sigma^2}{\lambda_m^2}\right)}{\sigma^{2m_0}\prod\limits_{m=m_0+1}^{\infty}\left(1-\frac{\sigma^2}{\sigma_m^2}\right)}
=\frac{\prod\limits_{m=n_0+1}^\infty\left(\frac{1}{\sigma^2}-\frac{1}{\lambda_m^2}\right)}{\prod\limits_{m=m_0+1}^\infty\left(\frac{1}{\sigma^2}-\frac{1}{\sigma_m^2}\right)}.
\]
In the purely formal sense, as $\sigma\to\infty$, this ratio tends to the constant
\begin{equation}\label{eq:C0defn}
C_0=C_{0;\Sigma,\Lambda}:=
\left\{
\begin{alignedat}{2}
&\prod\limits_{m=m_0+1}^{\infty}\frac{\sigma_m^2}{\lambda_m^2}&\qquad&\text{if }n_0=m_0,\\
(-1)^{n_0-m_0}\prod\limits_{m=m_0+1}^{n_0} \sigma_m^2&\prod\limits_{m=n_0+1}^{\infty}\frac{\sigma_m^2}{\lambda_m^2}&\qquad&\text{if }n_0>m_0,\\
(-1)^{m_0-n_0}\prod\limits_{m=n_0+1}^{m_0} \lambda_m^{-2}&\prod\limits_{m=m_0+1}^{\infty}\frac{\sigma_m^2}{\lambda_m^2}&\qquad&\text{if }n_0<m_0.
\end{alignedat}
\right.
\end{equation}
Note that in all the cases the infinite products in \eqref{eq:C0defn} are well-defined in view of \eqref{eq:quasibound} and Proposition \ref{prop:summary}(a).

The following result makes this formal calculation rigorous and also handles the singularities near the zeros.
\begin{theorem}\label{lem:goestozero} With terminology as above,
\[
\lim_{\sigma\to\infty} (Q_{\Lambda}(\sigma)-C_0Q_{\Sigma}(\sigma))=0.
\]
\end{theorem}

\begin{proof} Set
\begin{equation}\label{eq:MSigma}
\mathcal{M}_\Sigma(\sigma):=\left\{m\in\mathbb{N}: \sigma_m\in\Sigma, |\sigma-\sigma_m|\leq 1\right\}.
\end{equation}
By Proposition \ref{prop:summary}(b), there exists a constant $N\in\mathbb{N}$ such that for any $\sigma\ge 0$,
\[
\#\mathcal{M}_\Sigma(\sigma)\le N.
\]

We define a new function,
\[
\tilde{Q}_{\Sigma,\Lambda}(\sigma):=
Q_{\Sigma}(\sigma)\prod_{m\in\mathcal{M}_\Sigma(\sigma)}\frac{\lambda_m^2-\sigma^2}{\sigma_m^2-\sigma^2}=
Q_{\Sigma}(\sigma)\prod_{m\in\mathcal{M}_\Sigma(\sigma)}\frac{\lambda_m^2}{\sigma_m^2}\frac{\left(1-\frac{\sigma^2}{\lambda_m^2}\right)}{\left(1-\frac{\sigma^2}{\sigma_m^2}\right)}.
\]
This is, essentially, $Q_\Sigma(\sigma)$ but with the zeros near each fixed $\sigma$ moved to be the zeros of $Q_{\Lambda}(\sigma)$ instead. Observe that the product factor appearing in this definition has at most $N$ terms, and the whole expression can be also re-written, using  \eqref{eq:defofq}, as
\begin{equation}\label{eq:Qtildedef1}
\tilde{Q}_{\Sigma,\Lambda}(\sigma)=\sigma^{2n_0} \prod_{m\in\mathcal{M}_\Sigma(\sigma)}\frac{\lambda_m^2}{\sigma_m^2}\left(1-\frac{\sigma^2}{\lambda_m^2}\right)\prod_{m\not\in\mathcal{M}_\Sigma(\sigma)}\left(1-\frac{\sigma^2}{\sigma_m^2}\right).
\end{equation}

Then we claim that
\begin{equation}\label{eq:near}
\lim_{\sigma\to\infty} (\tilde{Q}_{\Sigma,\Lambda}(\sigma)-Q_{\Sigma}(\sigma))=0
\end{equation}
and
\begin{equation}\label{eq:far}
\lim_{\sigma\to\infty} (Q_{\Lambda}(\sigma)-C_0\tilde{Q}_{\Sigma,\Lambda}(\sigma))=0,
\end{equation}
from which the Theorem follows.

An observation that will be useful in the proofs of \eqref{eq:near} and \eqref{eq:far} is that since $Q_\Sigma(\sigma)$ is a multiple of $F(\sigma)$, it is uniformly bounded together with all the derivatives.

To prove \eqref{eq:near} we write
\begin{equation}\label{eq:near2}
\begin{split}
\tilde{Q}_{\Sigma,\Lambda}(\sigma)-Q_{\Sigma}(\sigma)&=-Q_{\Sigma}(\sigma)\left(1-\prod_{m\in\mathcal{M}_\Sigma(\sigma)}\frac{\lambda_m^2-\sigma^2}{\sigma_m^2-\sigma^2}\right)\\
&=-\frac{Q_\Sigma(\sigma)}{\prod\limits_{m\in\mathcal{M}_\Sigma(\sigma)}(\sigma_m-\sigma)}\cdot\frac{\prod\limits_{m\in\mathcal{M}_\Sigma(\sigma)}(\sigma_m^2-\sigma^2)-\prod\limits_{m\in\mathcal{M}_\Sigma(\sigma)}(\lambda_m^2-\sigma^2)}{\prod\limits_{m\in\mathcal{M}_\Sigma(\sigma)}(\sigma_m+\sigma)}\\
&=-P_1(\sigma)P_2(\sigma),
\end{split}
\end{equation}
where
\begin{align}
P_1(\sigma)&:=\frac{Q_\Sigma(\sigma)}{\prod\limits_{m\in\mathcal{M}_\Sigma(\sigma)}(\sigma_m-\sigma)},\\
\begin{split}
P_2(\sigma)&:=\frac{\prod\limits_{m\in\mathcal{M}_\Sigma(\sigma)}(\sigma_m^2-\sigma^2)-\prod\limits_{m\in\mathcal{M}_\Sigma(\sigma)}(\lambda_m^2-\sigma^2)}{\prod\limits_{m\in\mathcal{M}_\Sigma(\sigma)}(\sigma_m+\sigma)}\\
&=\prod_{m\in\mathcal{M}_\Sigma(\sigma)}(\sigma_m-\sigma) - \prod_{m\in\mathcal{M}_\Sigma(\sigma)}(\lambda_m-\sigma)\frac{\lambda_m+\sigma}{\sigma_m+\sigma}.\label{eq:P2}
\end{split}
\end{align}
We claim that $P_1(\sigma)$  is uniformly bounded and that $P_2(\sigma)$ tends to zero as $\sigma$ tends to infinity; this is enough  to establish \eqref{eq:near}.

To examine $P_1(\sigma)$ we note the following analysis fact:
\begin{proposition}\label{prop:analysisfact} 
For any function $f(x)$ on an interval $[a,b]$ which is $C^{k+1}$ and which is zero at $x_0\in[a,b]$,
\begin{equation}\label{eq:intid}
\left(\frac{f(x)}{x-x_0}\right)^{(k)}=\frac{\mathcal{I}_k(x)}{(x-x_0)^{k+1}},
\end{equation}
where
\[
\mathcal{I}_k(x):=\int\limits_{x_0}^x (t-x_0)^k f^{(k+1)}(t)\, \dr t.
\]
\end{proposition}

\begin{proof}[Proof of Proposition \ref{prop:analysisfact}]  We first note that
\begin{equation}\label{eq:Ikprime}
\mathcal{I}'_k(x)=(x-x_0)^k f^{(k+1)}(x),
\end{equation}
and, using integration by parts,
\begin{equation}\label{eq:Ikrec}
\begin{split}
(k+1)\mathcal{I}_k(x)&=\int\limits_{x_0}^x \frac{\dr (t-x_0)^{k+1}}{\dr t} f^{(k+1)}(t)\, \dr t\\
&=(x-x_0)^{k+1}f^{(k+1)}(x)-\mathcal{I}_{k+1}(x)
\end{split}
\end{equation}

We now prove \eqref{eq:intid} by induction in $k$. It is obviously true for $k=0$. Suppose now it holds for some $k$. Then,
using \eqref{eq:Ikprime} and \eqref{eq:Ikrec}, we obtain
\[
\begin{split}
\left(\frac{f(x)}{x-x_0}\right)^{(k+1)}&=\frac{\dr}{\dr x} \frac{\mathcal{I}_k(x)}{(x-x_0)^{k+1}}\\
&=-(k+1)(x-x_0)^{-k-2}\mathcal{I}_k(x)+(x-x_0)^{-1}f^{(k+1)}(x)\\
&=-(x-x_0)^{-1}f^{(k+1)}(x)+\mathcal{I}_{k+1}(x)+(x-x_0)^{-1}f^{(k+1)}(x)\\
&=\mathcal{I}_{k+1}(x).
\end{split}
\]
\end{proof}

Proposition \ref{prop:analysisfact} implies 
\begin{corollary}\label{cor:analysisfact} Under conditions of Proposition \ref{prop:analysisfact},
\begin{equation}\label{eq:analysisfact}
\left|\left(\frac{f(x)}{x-x_0}\right)^{(k)}(x)\right|\leq \left\|f^{(k+1)}\right\|_{C^{0}[a,b]}\quad\text{for all }x\in[a,b].
\end{equation}
\end{corollary}

\begin{proof}[Proof of Corollary \ref{cor:analysisfact}] We use \eqref{eq:intid}: the integrand in $\mathcal{I}_k(x)$ is bounded point-wise in absolute value by 
\[
|x-x_0|^{k}\cdot \left\|f^{(k+1)}\right\|_{C^0[a,b]},
\]
 from which \eqref{eq:analysisfact} follows.
\end{proof}

We now inductively apply  \eqref{eq:analysisfact} to $P_1(\sigma)$, taking $[a,b]=[\sigma-1,\sigma+1]$ and $x_0=\sigma_m\in\mathcal{M}_\Sigma(\sigma)$, to show that
\[
\left|P_1(\sigma)\right|\le \left\|Q_\Sigma(\sigma)\right\|_{C^{\#\mathcal{M}_\Sigma(\sigma)}[\sigma-1,\sigma+1]}.
\]
The right-hand side here is in turn is bounded by $\|Q_\Sigma(\sigma)\|_{C^N(\mathbb R)}$, which we know is finite. Thus, $P_1(\sigma)$ is uniformly bounded.

To analyse $P_2(\sigma)$, we use \eqref{eq:P2}.  There are at most $N$ terms in the sets $\mathcal{M}_\Sigma(\sigma)$, and therefore in the products in the right hand-side of \eqref{eq:P2}. All the elements of $\mathcal{M}_\Sigma(\sigma)$ go to $\infty$ as $\sigma\to\infty$.
Moreover, the absolute value of the difference of every two corresponding terms of these products,
\[
\left|(\sigma_m-\sigma)-(\lambda_m-\sigma)\frac{\lambda_m+\sigma}{\sigma_m+\sigma}\right|=\left|\frac{\sigma_m^2-\lambda_m^2}{\sigma_m+\sigma}\right|\le \frac{\left|\sigma_m^2-\lambda_m^2\right|}{2\sigma_m-1}
\]
goes to zero as $\sigma\to\infty$, and there is a uniform upper bound for all terms. By continuity of the product map from $\mathbb R^N$ to $\mathbb R$, the difference of products goes to zero, as desired. This completes the proof of \eqref{eq:near}.

We now proceed with establishing \eqref{eq:far}. For simplicity we assume from now on that the multiplicities of zero in sequences $\Sigma$ and $\Lambda$ coincide, that is $n_0=m_0$. The other cases can be treated in a similar manner.

In order to prove \eqref{eq:far} we intend to prove first that the function
\begin{equation}\label{eq:Rdef}
R_{\Sigma,\Lambda}(\sigma):=\ln\left|Q_{\Lambda}(\sigma)\right|-\ln\left|C_0\tilde{Q}_{\Sigma,\Lambda}(\sigma)\right|
\end{equation}
satisfies
\begin{equation}\label{eq:Rlim}
\lim_{\sigma\to\infty}R_{\Sigma,\Lambda}(\sigma)=0.
\end{equation}
Then \eqref{eq:far} follows immediately since the function $x\mapsto \er^x$ is uniformly continuous on any compact set, and both $Q_{\Lambda}(\sigma)$ and $\tilde{Q}_{\Sigma,\Lambda}(\sigma)$ are uniformly bounded on the positive real line by \eqref{eq:near} and  \eqref{eq:Rlim}, and both  terms in \eqref{eq:far} have the same sign for sufficiently large $\sigma$.

To prove \eqref{eq:Rlim}, we write out \eqref{eq:Rdef} explicitly using \eqref{eq:defofq}, \eqref{eq:Qtildedef1}, and \eqref{eq:C0defn}, and simplifying, yielding
\begin{equation}\label{eq:Rmore}
\begin{split}
R_{\Sigma,\Lambda}(\sigma)&=\sum_{m\not\in\mathcal{M}_\Sigma(\sigma)}\left(\ln\left|1-\frac{\sigma^2}{\lambda_m^2}\right|-\ln\frac{\sigma_m^2}{\lambda_m^2}-\ln\left|1-\frac{\sigma^2}{\sigma_m^2}\right|\right)\\
&=\sum_{m\not\in\mathcal{M}_\Sigma(\sigma)}\ln \left|\frac{\lambda_m^2-\sigma^2}{\sigma_m^2-\sigma^2}\right|
=\sum_{m\not\in\mathcal{M}_\Sigma(\sigma)}\left(\ln\left|\frac{\lambda_m-\sigma}{\sigma_m-\sigma}\right|+\ln\left|\frac{\lambda_m+\sigma}{\sigma_m+\sigma}\right|\right)\\
&=\sum_{m\not\in\mathcal{M}_\Sigma(\sigma)}\left(\ln \left|1+\frac{\lambda_m-\sigma_m}{\sigma_m-\sigma}\right|+\ln \left|1+\frac{\lambda_m-\sigma_m}{\sigma_m+\sigma}\right|\right).
\end{split}
\end{equation}

Each term of the sum in the right-hand side of \eqref{eq:Rmore} goes to zero as $\sigma\to\infty$, so any finite sum goes to zero. Let 
\begin{equation}\label{eq:tildeMSigma}
\widetilde{\mathcal{M}}_{\Sigma,\Lambda}(\sigma):=\mathcal{M}_{\Sigma}(\sigma)\cup\{m:\sigma_m\leq 1\}\cup\left\{m:|\lambda_m-\sigma_m|\geq \frac{1}{2}\right\}.
\end{equation}
and
\begin{equation}\label{eq:starMSigma}
\mathcal{M}^*_{\Sigma,\Lambda}(\sigma):=\mathbb{N}\setminus \widetilde{\mathcal{M}}_{\Sigma,\Lambda}(\sigma)=\left\{m: |\sigma_m-\sigma|>1, \sigma_m>1, |\lambda_m-\sigma_m|< \frac{1}{2}\right\};
\end{equation}
to write down the right-hand side of \eqref{eq:starMSigma} explicitly we have used \eqref{eq:MSigma} and \eqref{eq:tildeMSigma}.

Since $\widetilde{\mathcal{M}}_{\Sigma,\Lambda}(\sigma)\setminus \mathcal{M}_\Sigma(\sigma)$ is finite, we can replace the summation in the right-hand side of \eqref{eq:Rmore} by the sum over $m\in\mathcal{M}^*_{\Sigma,\Lambda}(\sigma)$. For those terms, 
\[
 \left|1+\frac{\lambda_m-\sigma_m}{\sigma_m\mp\sigma}\right|< \frac{1}{2},
 \]
and we use the fact that on the interval $\left[-\frac{1}{2},\frac{1}{2}\right]$ we have the inequality
\[
|\ln(1+x)|\leq 2|x|.
\]
Thus it suffices to show that the following expression goes to zero as $\sigma$ goes to infinity:
\begin{equation}\label{eq:sum2}
\sum_{m\in\mathcal{M}^*_{\Sigma,\Lambda}(\sigma)}\left( \left|\frac{\lambda_m-\sigma_m}{\sigma_m-\sigma}\right|+ \left|\frac{\lambda_m-\sigma_m}{\sigma_m+\sigma}\right|\right).
\end{equation}
The second term in \eqref{eq:sum2} is smaller than the first, and we have, by \eqref{eq:quasibound} and by Proposition \ref{prop:summary}(a),
\[
|\lambda_m-\sigma_m|\leq\const m^{-\epsilon}\leq\const \sigma_m^{-\epsilon},
\]
so it is enough to show decay of the expression
\[
R^*_{\Sigma,\Lambda}(\sigma):=\sum_{m\in\mathcal{M}^*_{\Sigma,\Lambda}(\sigma)}\frac{\sigma_m^{-\epsilon}}{|\sigma_m-\sigma|}.
\]
We have
\[
R^*_{\Sigma,\Lambda}(\sigma)\le R^\#_{\Sigma}(\sigma):=\sum_{m\in\mathcal{M}^\#_{\Sigma}(\sigma)}\frac{\sigma_m^{-\epsilon}}{|\sigma_m-\sigma|},
\quad\text{with}\quad\mathcal{M}^\#_{\Sigma}(\sigma):=\{m: |\sigma_m-\sigma|>1, \sigma_m>1\},
\]
and we will show that $R^\#_{\Sigma}(\sigma)\to 0$ as $\sigma\to\infty$ by comparing to an integral. Consider the function 
\[
g_\sigma(x):=x^{-\epsilon}|x-\sigma|^{-1}.
\]
 For each fixed $\sigma$, this function has no local maxima. So for each $\sigma_m$ with $m\in\mathcal{M}^\#_{\Sigma}(\sigma)$, there exists an interval $A_m$, of length $\frac{1}{2}$, either directly to the left or to the right of $\sigma_m$, for which 
\[
g_\sigma(\sigma_m)\leq 2\int_{A_m}g_\sigma(x)\, \dr x.
\]
We sum these inequalities over $m\in\mathcal{M}^\#_{\Sigma}(\sigma)$, and use the fact that each $x\in\mathbb R$ lies in at most $N$ such intervals $A_m$, to obtain the estimate
\begin{equation}\label{eq:Rhashbound}
R^\#_{\Sigma}(\sigma)=\sum_{m\in\mathcal{M}^\#_{\Sigma}(\sigma)} g_\sigma(\sigma_m)\leq 2N\int_{x>\frac{1}{2},|x-\sigma|\geq \frac{1}{2}}g_\sigma(x)\, \dr x.
\end{equation}

In principle, the integral in the right-hand side of \eqref{eq:Rhashbound} can be written down explicitly in terms of the incomplete beta functions, and the asymptotics as $\sigma\to\infty$ analysed, but the resulting expressions are pretty cumbersome,  so we instead break this integral into three parts to estimate. 
First consider
\[
\operatorname{Int}_1(\sigma):=\int_{\frac{1}{2}}^{\frac{\sigma}{2}}g_\sigma(x)\, \dr x=\int_{\frac{1}{2}}^{\frac{\sigma}{2}}\frac{x^{-\epsilon}}{\sigma-x}\, \dr x.
\]
The denominator is bounded below by $\sigma/2$, so we obtain a bound
\[
\operatorname{Int}_1(\sigma)\le \frac{2}{\sigma}\int_{\frac{1}{2}}^{\frac{\sigma}{2}}x^{-\epsilon}\, \dr x=
\begin{cases}
\frac{2^\epsilon}{\epsilon-1}\left(\sigma^{-1}-\sigma^{-\epsilon}\right)&\quad\text{if }\epsilon\ne 1,\\
2\frac{\ln\,\sigma}{\sigma}&\quad\text{if }\epsilon=1,
\end{cases}
\]
which goes to zero as $\sigma$ goes to infinity.

The second part is
\[
\operatorname{Int}_2(\sigma):=\int_{\frac{\sigma}{2}}^{\sigma-\frac{1}{2}}g_\sigma(x)\, \dr x=\int_{\frac{\sigma}{2}}^{\sigma-\frac{1}{2}}\frac{x^{-\epsilon}}{\sigma-x}\, \dr x.
\]
Here the numerator is bounded above by $\left(\frac{\sigma}{2}\right)^{-\epsilon}$, and we get a bound
\[
\operatorname{Int}_2(\sigma)\le \left(\frac{\sigma}{2}\right)^{-\epsilon}\int_{\frac{\sigma}{2}}^{\sigma-\frac{1}{2}}\frac{1}{\sigma-x}\, \dr x=\left(\frac{\sigma}{2}\right)^{-\epsilon}\ln\,\sigma,
\]
which again goes to zero as $\sigma$ goes to infinity.

The third and final part is
\[
\operatorname{Int}_3(\sigma):=\int_{\sigma+\frac{1}{2}}^\infty g_\sigma(x)\, \dr x=\int_{\sigma+\frac{1}{2}}^\infty\frac{x^{-\epsilon}}{x-\sigma}\, \dr x=\int_{\frac{1}{2}}^\infty\frac{1}{x(x+\sigma)^\epsilon}\, \dr x.
\]
The last integral goes to zero by the dominated convergence theorem, with the dominator $x^{-1-\epsilon}$, since the integrand converges to $0$ as $\sigma\to\infty$ for any fixed $x$. Thus all three integrals converge to zero, as $\sigma$ tends to infinity, completing the proof of \eqref{eq:far}, and with it the proof of Theorem \ref{lem:goestozero}.
\end{proof}

\begin{proof}[Proof of Theorem \ref{thm:A}] The upshot of Theorem  \ref{lem:goestozero} is that
\begin{equation}\label{eq:upshot}
F_{\balpha, \bell}(\sigma)=C_1 Q_{\Lambda}(\sigma)+o(1)\qquad\text{as}\quad\sigma\to+\infty,
\end{equation}
with some constant $C_1$. By repeating now word by word the construction of  Section \ref{subs:recovery} with $C$ replaced by $C_1$ and using \eqref{eq:Aofosmall}, we arrive at Theorem \ref{thm:A}.
\end{proof}
\section{Proof of Theorem \ref{thm:B}}
\subsection{Recovering the lengths sorted by magnitude} 
Assume, as in the statement of Theorem \ref{thm:B}, that we are given a characteristic polynomial, in the form \eqref{eq:Fdef2}, of an unknown curvilinear polygon $\mathcal{P}(\balpha,\bell)$ satisfying conditions  \eqref{eq:ellcond} (that is, the lengths are incommensurable over $\{0,\pm 1\}$) and \eqref{eq:alphacond} (all angles are not special). Recall that by \eqref{eq:Tdef}
\[
\mathcal{T}:=\{|\bell\cdot\bzeta|:  \bzeta\in\pmset{n}_+\},
\]
(and is known), and by \eqref{eq:calpha}
\[
\mathbf{c}:=\left(\cos\frac{\pi^2}{2\alpha_1},\dots,\cos\frac{\pi^2}{2\alpha_n}\right),
\]
(and is unknown). 
Set additionally 
\[
s_0:=\sign\left(\sin\frac{\pi^2}{2\alpha_1}\cdot\cdots\cdot\sin\frac{\pi^2}{2\alpha_n}\right)
\]
(yet unknown).

Then,
\begin{itemize}
\item all elements of $\mathcal{T}$ are distinct positive real numbers;
\item cardinality $\#\mathcal{T}$ is equal to $2^{n-1}$, and we can therefore immediately recover the number of vertices of $\mathcal{P}$ as
\[
n=\log_2(\#\mathcal{T})+1;
\]
\item the coefficients $r_k$,  $k=1,\dots,2^{n-1}$, are all non-zero (since there are no special angles), and their moduli do not exceed one;
\item $r_0\in(-1,1)$ is given by
\[
r_0=s_0\prod\limits_{j=1}^n \sqrt{1-c_j^2}.
\]
\end{itemize}

Assume, as above, that we are given a trigonometric polynomial in the form \eqref{eq:Fdef2} corresponding to a non-special polygon with incommensurable lengths.
Then we have the following
\begin{theorem}\label{thm:inverse2}
Given a set of frequencies $\mathcal{T}=\mathcal{T}_n$, we can reconstruct a permutation of the vector of lengths
\[
\bell'=(\ell_1',\dots,\ell'_n),
\]
such that the lengths are sorted increasingly,
\[
\ell_1'<\ell_2'<\dots<\ell_n'.
\]
\end{theorem}

\begin{proof}[Proof of Theorem \ref{thm:inverse2}] Without loss of generality, re-order the frequencies in an increasing order,
\[
t_1<t_2<\dots<t_{2^{n-1}}.
\]
We now proceed in steps, where on Step $k$, $k=1,\dots,n$, we determine the value of $\ell_k'$.

\emph{Step $0$}. We immediately have
\[
L:=\ell_1+\cdots+\ell_n=\ell'_1+\cdots+\ell'_n=\max \mathcal{T}_n= t_{2^{n-1}}.
\]

\emph{Step $1$}. Set
\[
\mathcal{T}_{n,1}:=\mathcal{T}_n \setminus \{L\}.
\]
Then
\[
\ell_1'=\frac{1}{2}(L-\max\mathcal{T}_{n,1})=\frac{1}{2}(L-t_{2^{n-1}-1}).
\]

\emph{Step $k$, $k=2,\dots,n-1$}. Suppose we have already found  $\ell_1'<\dots<\ell'_{k-1}$. Set
\[
\tilde{\mathcal{T}}_{n,k}:=\left\{\pm\left(L-2\sum_{j=1}^{k-1}f_j t_j\right)\middle| f_j\in\{0,1\}\right\},
\]
and
\[
\mathcal{T}_{n,k}:=\mathcal{T}_n \setminus \tilde{\mathcal{T}}_{n,k}
\]
(basically, to obtain $\mathcal{T}_{n,k}$ we exclude from $\mathcal{T}_{n}$ all linear combinations of lengths which may have minuses in front of already found $\ell_1',\dots,\ell'_{k-1}$ and do not have minuses anywhere else, and the negations of such linear combinations). Then,
\[
\ell_k'=\frac{1}{2}(L-\max\mathcal{T}_{n,k}).
\]

\emph{Step $n$}. Set
\[
\ell_n' = L-\sum\limits_{j=1}^{n-1} \ell'_j.
\]

Thus, we recover the lengths $\ell'_1<\dots<\ell'_n$. 
\end{proof}

We do not know yet the original order of the sides, that is, the permutation $\left(m_k\right)_{k=1}^n$ such that $\ell'_k=\ell_{m_k}$. We will  also use the inverse permutation $\left(k_m\right)_{m=1}^n$ such that $\ell'_{k_m}=\ell_m$.
\subsection{Recovering the correct order of the sides and the information on the angles}
\label{subsec:proofthmb}
 Once we found the vector $\bell'$, we still need to determine the correct order of sides and the angles, with appropriate modifications in the exceptional case. We consider separately three cases.

{\bf Case $n=1$}. Then there is only one angle and one side, the trigonometric polynomial has the form $\cos(t_1 \sigma)+r_0$, where $t_1=\ell_1$, and we know the coefficient $r_0=s_0\sqrt{1-c_1^2}$. We have $\bell=(t_1)$. There are two sub-cases:
\begin{description}
\item[$r_0=0$:] Then the only angle is exceptional,  therefore $\mathbf{c}=\pm(1)$, and the exceptional  boundary component is even.
\item[$r_0\ne 0$:] The angle is non-exceptional, and $\mathbf{c}=\pm\left(\sqrt{1-r_0^2}\right)$ (as $|s_0|=1$).
\end{description}

{\bf Case $n=2$}. There are two angles, and the trigonometric polynomial has the form $r_1\cos(t_1 \sigma)+\cos(t_2\sigma)-r_0$, where 
\begin{equation}\label{eq:case2}
r_1=c_1c_2,\qquad r_0^2=(1-c_1^2)(1-c_2^2).
\end{equation} 
Without loss of generality $\bell'=\bell$. There are three sub-cases:
\begin{description}
\item[$r_0=0$ and $|r_1|=1$:] Both angles are exceptional, and there are two exceptional boundary components of one side each.  They are both even if $r_1=1$ and both odd if $r_1=-1$.
\item[$r_0=0$ and $|r_1|<1$:] One angle is exceptional and another is non-exceptional, and there is one even exceptional boundary component. Without loss of generality we can assume that $\alpha_1$ is even exceptional  ($c_1=1$), then $c_2=r_1$, and therefore allowing for a change of orientation and a change of sign, $\mathbf{c}=\pm(1,r_1)$.
\item[$r_0\ne 0$:] Both angles are non-exceptional.  Solving the quadratic equations deduced from \eqref{eq:case2}, we obtain
\[
\{|c_1|,|c_2|\}=\left\{\rho, \frac{r_1}{\rho}\right\},\qquad\text{with }\rho=\frac{\sqrt{1+r_1^2-r_0^2+\sqrt{(1+r_1^2-r_0^2)^2-4r_1^2}}}{\sqrt{2}},
\]
and therefore allowing for a change of orientation and a change of sign, $\mathbf{c}=\pm\left(\rho, \frac{r_1}{\rho}\right)$.
\end{description}

{\bf Case $n>2$}. Once we know all the $\ell_j'$, we know the linear combination of $\pm \ell_j'$  which corresponds to a particular frequency $t_k$. In order to proceed further, we need to re-write the characteristic trigonometric polynomial once more using a slightly different notation.

First, consider a subset $\mathcal{J}\subseteq \{1,\dots,n\}$. We denote by $\bzeta(\mathcal{J})=(\zeta_1,\dots,\zeta_n)$ a vector in  $\pmset{n}$ such that 
\[
\zeta_j=\begin{cases} 1,&\quad\text{if }j\in\mathcal{J},\\-1,&\quad\text{if }j\not\in\mathcal{J}.\end{cases}
\]
For every such subset  $\mathcal{J}$ there exists a unique element $t:=t(\mathcal{J})\in\mathcal{T}_n$ such that $t=|\bzeta(\mathcal{J})\cdot \bell'|$. We note also that $t(\mathcal{J})=t(\{1,\dots,n\}\setminus\mathcal{J})$. The characteristic polynomial can be re-written as
\[
\sum_{\mathcal{J}\subseteq \{1,\dots,n\}} \frac{r(\mathcal{J})}{2} \cos(t(\mathcal{J}) \sigma)-r_0,
\]
where all the amplitudes $r(\mathcal{J})$ are known. 

It will be in particular useful to write down the amplitudes $r(\mathcal{J})$ in cases when a subset $\mathcal{J}$ contains either one or two elements. For $\mathcal{J}=\{k\}$, the vector $\bzeta(\mathcal{J})$ will have exactly \emph{two} sign changes, in positions $m_{k-1}$ and $m_k$, and we have
\[
r'_k:=r(\{k\})=c_{m_k-1} c_{m_k},
\]
see Figure \ref{fig:fig2}(a).

\begin{figure}[hbt]
\begin{center}
\includegraphics[width=2.5in]{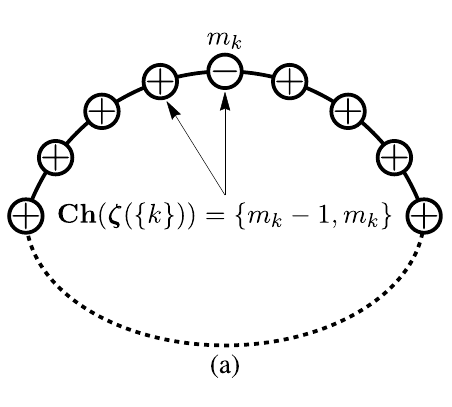}\qquad \includegraphics[width=2.5in]{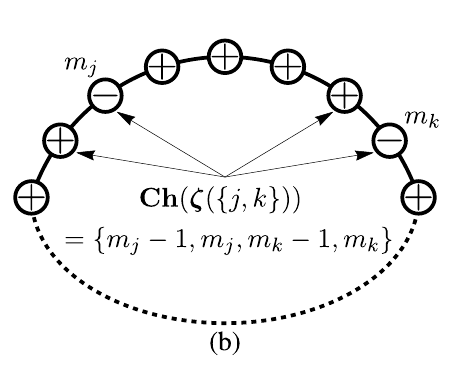}\ \includegraphics[width=2.5in]{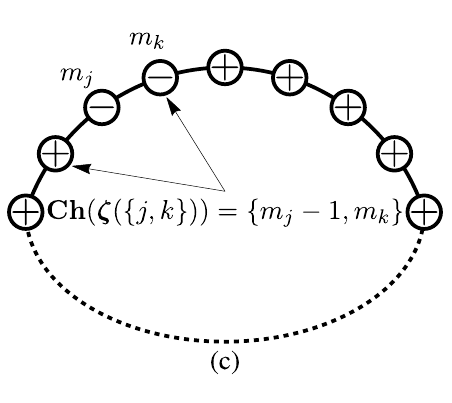}
\caption{Illustration of the sign changes in $\bzeta(\mathcal{J})$ in three cases: (a) $\mathcal{J}=\{k\}$; (b) $\mathcal{J}=\{j,k\}$, and the vertices $V_{m_j}$ and $V_{m_k}$ are not neighbours; (c) $\mathcal{J}=\{j,k\}$, and the vertices $V_{m_j}$ and $V_{m_k}$ are neighbours, in the case shown with $m_j=m_k-1$.\label{fig:fig2}}
\end{center}
\end{figure}

For $\mathcal{J}=\{j,k\}$, with $j\ne k$, the situation is more complicated and depends on whether the vertices $V_{m_j}$ and $V_{m_k}$ are \emph{neighbours}, that is, on whether $|m_j-m_k|=1$.  If they are \emph{not} neighbours, there are \emph{four} sign changes in $\bzeta(\mathcal{J})$, in positions $m_{j-1}$, $m_j$, $m_{k-1}$, and $m_k$, and we have
\[
r(\{j,k\}=r'_jr'_k,
\]
see Figure \ref{fig:fig2}(b).

If the vertices  $V_{m_j}$ and $V_{m_k}$ \emph{are} neighbours, then  the vector $\bzeta(\mathcal{J})$ will have again exactly \emph{two} sign changes, now at positions $\min(m_j,m_k)-1$ and $\max(m_j,m_k)$, and we have
\[
r(\{j,k\})=c_{\min(m_j,m_k)-1} c_{\max(m_j,m_k)}=\frac{r'_jr'_k}{c^2_{\min(m_j,m_k)}},
\]
see Figure \ref{fig:fig2}(c).

We emphasise that at this stage we do not yet know the correct enumerating sequence $m_k$.  On the other hand we know the matrix
\[
R'_{j,k}:=r(\{j,k\}),\qquad j,k=1,\dots,n;
\]
its diagonal entries are $R'_{k,k}=r'_k$.  Introduce additionally the matrix
\[
D'_{j,k}:=\frac{R'_{j,j}R'_{k,k}}{R'_{j,k}},\qquad j,k=1,\dots,n.
\]
This matrix is symmetric, and its off-diagonal entries 
(which are all positive) indicate which sides are adjacent to each other, in the following sense:
\begin{itemize}
\item If $D'_{j,k}<1$ for some $j\ne k$, then the sides with lengths $\ell'_j$ and $\ell'_k$ are adjacent to each other, and for the angle $\alpha_p$ between them (with $p=\max(m_j, m_k)$) the corresponding element $c_p$ of the vector $\mathbf{c}$ can be found, up to sign:  
\[
\left|c_p\right|=\left|\cos\frac{\pi^2}{2\alpha_p}\right|=\sqrt{D'_{j,k}}.    
\]
\item If $D'_{j,k}=1$ for some $j\ne k$, then the corresponding sides with lengths $\ell'_j$ and $\ell'_k$ are either not adjacent, or are adjacent but with an exceptional angle between them.
\end{itemize}
 
We can now use the properties of the matrix $D'$ to find, first, the number $K$ of exceptional angles. Note that in the non-exceptional case each row of  $D'$ contains exactly \emph{two} off-diagonal entries which are less than one. In the exceptional case, a row number $j$ may have \emph{one} such entry (which indicates that there is an exceptional angle at \emph{one} end of the side $\ell'_j$) or \emph{zero} such entries (indicating that  there are exceptional angles at \emph{both} ends of this side). Thus, we can recover the number of exceptional angles as
\begin{equation}\label{eq:Kdeduced}
K=n-\frac{\#\{(j,k): j\ne k\text{ and } D'_{j,k}<1\}}{2}.
\end{equation}

Assuming for the moment that $K=0$, we can now proceed with determining the side-lengths $\bell$ in the correct order, and the vector $\pm\mathbf{c}$. 
From now on, without loss of generality we can assume $m_1=1$, so that $\ell_1=\ell'_1$.
By inspection of the first row of matrix $D'$ we can find two indices, denoted $k_2$ and $k_n$, such that entries $D_{1,k_2}$ and $D_{1,k_n}$ are strictly less than $1$. Therefore, the sides $\ell'_{k_2}$ and $\ell'_{k_n}$ are neighbours of the side $\ell_1=\ell'_1$, and should be re-labelled as $\ell_2$ and $\ell_n$ (we have the freedom of choosing enumeration of these two sides at this stage, hence an ambiguity in choosing the orientation). Suppose, for definiteness, that  we set $m_{k_2}=2$ and $m_{k_n}=n$.
Then we have, for the angle between $\ell_1$ and $\ell_2$, $|c_1|=\sqrt{D'_{1,k_2}}$, and for the angle  between $\ell_1$ and $\ell_n$, $|c_n|=\sqrt{D'_{1,k_n}}$.
 
 We now continue the process by looking at the row number $k_2$ of $D'$. We have already determined one of the entries in this row which is less than one: it is $D'_{k_2,1}$ (by the symmetry of $D'$). Let the index of the other such entry be denoted by $k_3$. Then the side $\ell_3:=\ell'_{k_3}$ is adjacent to $\ell_2$, and we set $m_{k_3}=3$ and
 find, for the angle between $\ell_2$ and $\ell_3$, $|c_2|=\sqrt{D'_{k_2,k_3}}$.  
 
Continuing the process, we determine the order of all sides (modulo reversal of orientation), and the vector $(|c_1|,\dots,|c_n|)$.

In the presence of exceptional angles ($K>0$), we proceed in a similar manner with the following modifications. We start the process at a row of $D'$ in which there is 
exactly one off-diagonal entry which is less than one, if such a row exists (otherwise choose a row with no  off-diagonal entries less than one). Assume it is the first row of $D$ and set, for the first exceptional boundary component, $\ell^{(1)}_1=\ell'_1$. Suppose $D'_{1,k_2}<1$; then set $\ell^{(1)}_2=\ell'_{k_2}$ and $\left|c^{(1)}_1\right|=\sqrt{D'_{1,k_2}}$. We continue the process until we reach a row of $D'$ in which no further off-diagonal entry less than one can be found. We then re-start the process from another (as yet unencountered) row of $D'$ to find the second exceptional boundary component, and so on.

To finish the proof of Theorem  \ref{thm:B} it remains only to show, in the non-exceptional case, that if we fix the sign of the cosine $c_1$, say, the signs of other  cosines $c_2,\dots,c_n$ will be determined automatically. This in fact follows immediately: the angles $\alpha_{m}$ and $\alpha_{m-1}$ are adjacent to the side $\ell_m=\ell'_{k_m}$, and therefore
\begin{equation}\label{eq:csign}
\sign\left(c_{m-1}c_m\right)=\sign\left(D_{k_m,k_m}\right). 
\end{equation}
The exceptional case is dealt with similarly. \qed
\begin{remark}
\label{rem:quantum}
In view of  Remark \ref{quantum0}, a combination of  Theorems \ref{prop:Aprime} and  \ref{thm:B} may be perceived as an inverse spectral result for a certain special family of quantum graphs.  Namely, let $\mathcal{G}=\mathcal{G}_{\balpha,\bell}$ be a circular graph with $n$ vertices $V_1,\dots,V_n$ enumerated clock-wise, with edges of length $\ell_j$ joining $V_{j-1}$ and $V_j$. Let $s\in[0,L]$ be a global edge length variable of $\mathcal{G}$. We consider the following spectral problem on $\mathcal{G}$:
\begin{equation}\label{eq:G}
\begin{split}
-\frac{\mathrm{d}^2 f}{\mathrm{d}^2 s}&=\nu f,\\ 
\sin\frac{\pi^2}{4\alpha_j} f|_{V_j-0}&=\cos\frac{\pi^2}{4\alpha_j}f|_{V_j+0},\\
\cos\frac{\pi^2}{4\alpha_j} f'|_{V_j-0}&=\sin\frac{\pi^2}{4\alpha_j}f'|_{V_j+0}.
\end{split}
\end{equation} 
Then, according to \cite[Theorem 2.24]{LPPS}, the eigenvalues $\nu_m$ of \eqref{eq:G} are related to quasi-eigenvalues $\sigma_m$ as
\[
\sqrt{\nu_m}=\sigma_m
\] 
with account of multiplicities.

Our methods allow one to  recover from the spectrum of a quantum graph not only its  edge lengths $\bell$  
(which is expected, see, for example,  \cite{KoSm, GuSmi, KN1, KostSch06, KostPotSch, Bol, KN2})  
but also the information on vertex matching conditions encoded by the vector $\pm\mathbf{c}_{\balpha}$: indeed, if we know the quantities $\pm c_j=\pm \cos\frac{\pi^2}{2\alpha_j}$, then we know the quantities 
\[
\left\{\left|\tan\frac{\pi^2}{4\alpha_j}\right|, \left|\cot\frac{\pi^2}{4\alpha_j}\right|\right\}=\left\{\sqrt{\frac{1\mp c_j}{1\pm c_j}}\right\}
\] 
which determine the vertex conditions in \eqref{eq:G} up to a change of sign and a change of orientation. 

\end{remark}
\subsection{Examples}
\label{subsec:examples}
\begin{example}\label{ex:recovernonex}
Consider the trigonometric polynomial
\begin{equation}\label{eq:ex1F}
\begin{split}
F(\sigma)&=-\frac{1}{60}\cos((2+\er-\pi)\sigma)
-\frac{1}{3}\cos((2\sqrt{2}+\er-\pi)\sigma)
+\frac{1}{8}\cos((2-\er+\pi)\sigma)\\
&-\frac{2}{15}\cos((-2\sqrt{2}+\er+\pi)\sigma)
+\frac{1}{10}\cos((2\sqrt{2}-\er+\pi)\sigma)
-\frac{1}{6}\cos((-2+\er+\pi)\sigma)\\
&+\frac{1}{20}\cos((2+\er+\pi)\sigma)
+\cos((2\sqrt{2}+\er+\pi)\sigma)
-\sqrt{\frac{3}{8}}.
\end{split}
\end{equation}
The cosine terms are ordered in increasing order of frequencies $t_k$, $k=1,\dots,8$.

We start by finding the side-lengths, in increasing order, following the procedure in the proof of Theorem  \ref{thm:inverse2}.

\emph{Steps $0$ and $1$}. By inspection, we immediately have
\[
L=\sum_{k=1}^4 \ell'_k=2\sqrt{2}+\er+\pi,
\]
and
\[
\ell_1'=\frac{1}{2}(L-t_7)=\frac{1}{2}(L-(2+\er+\pi))=\sqrt{2}-1.
\]

\emph{Step $2$}. We have
\[
\mathcal{T}_{4,2}=\{ t_1,\dots, t_6\},
\]
and therefore
\[
\ell_2'=\frac{1}{2}(L-t_6)=\frac{1}{2}(L-(-2+\er+\pi))=1+\sqrt{2}.
\]

\emph{Step $3$}. We have
\[
\mathcal{T}_{4,3}=\{t_1,t_2,t_3,t_5\},
\]
and therefore
\[
\ell_3'=\frac{1}{2}(L-t_5)=\frac{1}{2}(L-(2\sqrt{2}-\er+\pi))=\er.
\]

\emph{Step $4$}. Finally,
\[
\ell_4'=L-\ell_1'-\ell_2'-\ell_3'=\pi.
\]

We now proceed to determine the order of sides $\ell_{m_k}=\ell'_k$ in the polygon, and the corresponding quantities $|c_k|$.  We re-write \eqref{eq:ex1F} as
\[
\begin{split}
F(\sigma)&=-\frac{1}{60}\cos\left((-++-)\cdot\bell'\,\sigma\right)
-\frac{1}{3}\cos\left((+++-)\cdot\bell'\,\sigma\right)
+\frac{1}{8}\cos\left((-+-+)\cdot\bell'\,\sigma\right)\\
&-\frac{2}{15}\cos\left((--++)\cdot\bell'\,\sigma\right)
+\frac{1}{10}\cos\left((++-+)\cdot\bell'\,\sigma\right)
-\frac{1}{6}\cos\left((+-++)\cdot\bell'\,\sigma\right)\\
&+\frac{1}{20}\cos\left((-+++)\cdot\bell'\,\sigma\right)
+\cos\left((++++)\cdot\bell'\,\sigma\right)
-\sqrt{\frac{3}{8}}.
\end{split}
\]
By inspection, the matrix $R'$ is
\[
R'=\begin{pmatrix}
 \frac{1}{20} & -\frac{2}{15} & \frac{1}{8} & -\frac{1}{60} \\
 -\frac{2}{15} & -\frac{1}{6} & -\frac{1}{60} & \frac{1}{8} \\
 \frac{1}{8} & -\frac{1}{60} & \frac{1}{10} & -\frac{2}{15} \\
 -\frac{1}{60} & \frac{1}{8} & -\frac{2}{15} & -\frac{1}{3}
 \end{pmatrix},
 \]
 and therefore the matrix $D'$ is
 \[
 D'=
 \begin{pmatrix}
 \frac{1}{20} & \frac{1}{16} & \frac{1}{25} & 1 \\
 \frac{1}{16} & -\frac{1}{6} & 1 & \frac{4}{9} \\
 \frac{1}{25} & 1 & \frac{1}{10} & \frac{1}{4} \\
 1 & \frac{4}{9} & \frac{1}{4} & -\frac{1}{3}
\end{pmatrix}
\]

Set $\ell_1=\ell_1'=\sqrt{2}-1$. Looking for the entries different from one in the first row of $D'$, we set $k_2=2$, $k_4=3$ (and so $m_2=2$, $m_3=4$), and therefore obtain 
\[
\ell_2=\ell'_2=1+\sqrt{2},\qquad |c_1|=\sqrt{D'_{1,2}}=\frac{1}{4},
\]
and
\[
\ell_4=\ell'_3=\er,\qquad |c_4|=\sqrt{D'_{1,3}}=\frac{1}{5}.
\]

Switching to the second row of $D'$, we set $k_3=4$ (and so $m_4=3$), and further obtain 
\[
\ell_3=\ell'_4=\pi,\qquad |c_2|=\sqrt{D'_{2,4}}=\frac{2}{3},
\]
and finally
\[
|c_3|=\sqrt{D'_{4,3}}=\frac{1}{2}.
\]
Summarising, we have so far
\[
\begin{split}
(m_k)=(1,2,4,3),\qquad (k_m)&=(1,2,4,3),\\
(|c_1|, |c_2|, |c_3|, |c_4|)&=\left(\frac{1}{4}, \frac{2}{3}, \frac{1}{2},  \frac{1}{5}\right),
\end{split}
\]
and
\begin{equation}\label{eq:lfound}
\bell=(\ell'_1,\ell'_2,\ell'_4,\ell'_3)=(\sqrt{2}-1, 1+\sqrt{2},\pi,\er).
\end{equation}
Finally, to find the signs of $c_j$, we use \eqref{eq:csign}, giving
\[
\sign(c_1 c_2)=\sign(D'_{2,2})=-1,\quad \sign(c_2 c_3)=\sign(D'_{4,4})=-1,\quad \sign(c_3 c_4)=\sign(D'_{3,3})=1,
\] 
and so
\begin{equation}\label{eq:cfound}
\mathbf{c}=\pm\left(\frac{1}{4}, -\frac{2}{3}, \frac{1}{2},  \frac{1}{5}\right).
\end{equation}
We remark that the trigonometric polynomial \eqref{eq:ex1F} was in fact constructed as $F_{\tilde{\balpha}, \tilde{\bell}}(\sigma)$ with 
\begin{equation}\label{eq:lgiven}
\tilde{\bell}=\left(\er,\pi,1+\sqrt{2},\sqrt{2}-1\right),
\end{equation}
and $\tilde{\balpha}$ such that 
\begin{equation}\label{eq:cgiven}
\tilde{\mathbf{c}}=\mathbf{c}_{\tilde{\balpha}}=\left(\frac{1}{2}, -\frac{2}{3}, \frac{1}{4}, \frac{1}{5}\right).
\end{equation}
Comparing \eqref{eq:lfound}--\eqref{eq:cfound} with  \eqref{eq:lgiven}--\eqref{eq:cgiven}, we can confirm that we have indeed recovered the geometric information about the polygon
within the restrictions of Theorem \ref{thm:B}(a), see also Remark \ref{ftn:c}.  
 \end{example}

\begin{example}\label{ex:recoverex}
Consider now the trigonometric polynomial
\begin{equation}\label{eq:ex2F}
\begin{split}
F(\sigma):&=F_{\tilde{\balpha}_\mathrm{ex},\tilde{\bell}}(\sigma)=-\frac{1}{2}\cos((2+\er-\pi)\sigma)
-\frac{1}{2}\cos((2\sqrt{2}+\er-\pi)\sigma)\\
&+\frac{1}{2}\cos((2-\er+\pi)\sigma)
-\cos((-2\sqrt{2}+\er+\pi)\sigma)
-\frac{1}{2}\cos((2\sqrt{2}-\er+\pi)\sigma)\\
&+\cos((-2+\er+\pi)\sigma)
-\cos((2+\er+\pi)\sigma)
+\cos((2\sqrt{2}+\er+\pi)\sigma).
\end{split}
\end{equation}
It is generated using the same vector $\tilde{\bell}$  (given by \eqref{eq:lgiven}) as in Example \ref{ex:recovernonex} (and therefore has the same frequencies $t_k$, $k=1,\dots,8$) but with $\tilde{\balpha}$ replaced by a vector $\tilde{\balpha}_\mathrm{ex}$ such that   
\begin{equation}\label{eq:cgivenex}
\tilde{\mathbf{c}}_\mathrm{ex}=\mathbf{c}_{\tilde{\balpha}_\mathrm{ex}}=\left(\frac{1}{2}, 1, 1, -1\right).
\end{equation} 
We will now use the procedure outlined in the exceptional case of Theorem \ref{thm:B}(b) to recover the geometric information from \eqref{eq:ex2F}.

Since the frequencies are the same as in Example \ref{ex:recoverex}, the recovery of the vector $\bell'$ goes exactly as before. The matrices $R'$ and $D'$ become, by inspection,
\[
R'=\begin{pmatrix}
-1 & -1 & \frac{1}{2} & -\frac{1}{2} \\
 -1 & 1 & -\frac{1}{2} & \frac{1}{2} \\
 \frac{1}{2} & -\frac{1}{2} & -\frac{1}{2} & -1 \\
 -\frac{1}{2} & \frac{1}{2} & -1 & \frac{1}{2}
 \end{pmatrix}
 \qquad\text{and}\qquad
 D'=\begin{pmatrix}
-1 & 1 & 1 & 1 \\
 1 & 1 & 1 & 1 \\
 1 & 1 & -\frac{1}{2} & \frac{1}{4} \\
 1 & 1 & \frac{1}{4} & \frac{1}{2}
  \end{pmatrix}.
  \]
  
Formula \eqref{eq:Kdeduced} then gives the number of exceptional angles $K=3$. Starting the reconstruction of exceptional boundary components with the third row of $D'$,
we set $\ell^{(1)}_1=\ell'_3=\er$. Since the only non-unity off-diagonal entry in this line is $D'_{3,4}$, we set $\ell^{(1)}_2=\ell'_4=\pi$ and $\left|c^{(1)}_1\right|=\sqrt{D'_{3,4}}=\frac{1}{2}$. There are no further sides connected to $\ell^{(1)}_2$, therefore the first exceptional boundary component has two sides; as $\sign(D'_{3,3}D'_{4,4})=-1$, this exceptional boundary component is odd.

The other two exceptional boundary components have only one arc each; we can choose $\bell^{(2)}=(\ell'_1)=\left(\sqrt{2}-1\right)$ (this exceptional boundary component being odd since $\sign(D'_{1,1})=-1$), and  $\bell^{(3)}=(\ell'_2)=\left(\sqrt{2}+1\right)$ (this exceptional boundary component being even since $\sign(D'_{2,2})=+1$).
\end{example}

The next four examples illustrate that Theorem \ref{thm:B} no longer holds if we drop either the condition that there are no special angles, or the condition of sides incommensurability. Note that if the incommensurability condition fails, then the cardinality of the frequencies set  \eqref{eq:Tdef} is strictly less than $2^{n-1}$, while in the presence of the special angles  some amplitudes vanish, and therefore certain frequencies do not show up in the trigonometric polynomial.  In both situations we are therefore unable to follow the recovery procedure of Theorem \ref{thm:inverse2}.

\begin{example}\label{ex:special1}  Consider a family of straight parallelograms 
\[
P_a:=\mathcal{P}\left(\left(\frac{\pi}{5},\frac{4\pi}{5},\frac{\pi}{5},\frac{4\pi}{5}\right), (a,1-a,a,1-a)\right)
\]
depending on a parameter $0<a<1$, with sides $a$ and $1-a$ (and therefore a fixed perimeter $L=2$), and angles $\frac{\pi}{5}$ (which is special) and $\frac{4\pi}{5}$. 
Then the characteristic polynomial \eqref{eq:Fdef} for $P_a$ is
\[
F(\sigma)=\cos(2\sigma)-\frac{1}{\sqrt{2}}.
\]
As it is independent of $a$, we cannot recover the side-lengths from it. Note that in this example, \emph{both} conditions \eqref{eq:ellcond}  and \eqref{eq:alphacond} are not satisfied. 
\end{example}

\begin{example}\label{ex:special2} We additionally show that if we allow special angles, there exist pairs of straight triangles (with pairwise different vectors $\bell$ and $\mathbf{c}$) which produce identical trigonometric polynomials \eqref{eq:Fdef}. For $i_1, i_2\in\mathbb{N}$, consider a triangle $T_{i_1,i_2}$ with perimeter one,  two special angles $\alpha_1=\frac{\pi}{2i_1+1}$, $\alpha_2=\frac{\pi}{2i_2+1}$, and the third angle $\alpha_3=\pi-\alpha_1-\alpha_2$. Then \eqref{eq:Fdef} for $T_{i_1,i_2}$ becomes, after some simplifications,
\[
F(\sigma)=\cos(\sigma)+(-1)^{i_1+i_2} \sin 2\pi\frac{i_1+i_2+1}{4i_1i_2-1}.
\]
The polynomials for two different triangles $T_{i_1,i_2}$ and $T_{\tilde{i}_1,\tilde{i}_2}$ would coincide if they have the same constant term, in particular if
\begin{equation}\label{eq:ms}
(-1)^{i_1+i_2}\frac{i_1+i_2+1}{4i_1i_2-1}=(-1)^{\tilde{i}_1+\tilde{i}_2}\frac{\tilde{i}_1+\tilde{i}_2+1}{4\tilde{i}_1\tilde{i}_2-1}
\end{equation}
has a solution $(i_1,i_2,\tilde{i}_1,\tilde{i}_2)\in\mathbb{N}^4$.

One solution of \eqref{eq:ms} is given by $(i_1,i_2,\tilde{i}_1,\tilde{i}_2)=(3,31,4,10)$; thus two triangles with perimeter one and angles $\balpha=\left(\frac{\pi}{7}, \frac{\pi}{63}, \frac{53\pi}{63}\right)$ and
$\tilde{\balpha}=\left(\frac{\pi}{9}, \frac{\pi}{21}, \frac{53\pi}{63}\right)$, respectively, are indistinguishable from their respective Steklov quasi-eigenvalues.
\end{example}

\begin{example}\label{ex:comm1}
Consider two curvilinear triangles $\mathcal{Q}=\mathcal{P}(\balpha, \bell)$ and $\tilde{\mathcal{Q}}=\mathcal{P}(\tilde{\balpha}, \tilde{\bell})$, with $\bell=(1,1,3)$ and $\tilde{\bell}=(1,2,2)$, so that  \eqref{eq:ellcond} is not satisfied. Then we claim that the angles  $\balpha=(\alpha_1,\alpha_2,\alpha_3)$ and $\tilde{\balpha}=(\tilde{\alpha}_1,\tilde{\alpha}_2,\tilde{\alpha}_3)$ can be chosen in such a way that
\begin{equation}\label{eq:TT'}
F_{\balpha, \bell}(\sigma)=F_{\tilde{\balpha}, \tilde{\bell}}(\sigma),
\end{equation}
for all $\sigma\in\mathbb{R}$, and therefore all the quasi-eigenvalues of $\mathcal{Q}$ and $\tilde{\mathcal{Q}}$ coincide.

Set $c_j:=\cos\frac{\pi^2}{2\alpha_j}$,  $s_j:=\sin\frac{\pi^2}{2\tilde{\alpha}_j}$, $\tilde{c}_j:=\cos\frac{\pi^2}{2\tilde{\alpha}_j}$, and $\tilde{s}_j:=\sin\frac{\pi^2}{2\tilde{\alpha}_j}$, $j=1,2,3$. We now write down \eqref{eq:TT'} explicitly using the definitions, yielding
\[
\begin{split}
&\cos(5\sigma)+c_2c_3\cos(\sigma)+c_1c_2\cos(3\sigma)+c_1c_3\cos(3\sigma)-s_1s_2s_3\\=
&\cos(5\sigma)+\tilde{c}_2\tilde{c}_3\cos(\sigma)+\tilde{c}_1\tilde{c}_2\cos(\sigma)+\tilde{c}_1\tilde{c}_3\cos(3\sigma)-\tilde{s}_1\tilde{s}_2\tilde{s}_3,
\end{split}
\]
which becomes an identity if we can find an instance of
\begin{equation}\label{eq:cssystem}
\begin{cases}
c_2c_3&=\tilde{c}_2\tilde{c}_3+\tilde{c}_1\tilde{c}_2,\\
c_1c_2+c_1c_3&=\tilde{c}_1\tilde{c}_3,\\
s_1s_2s_3&=\tilde{s}_1\tilde{s}_2\tilde{s}_3.
\end{cases}
\end{equation}
It is easily checked that the system \eqref{eq:cssystem} is satisfied, for example, if we choose
\[
c_1=c_2=\tilde{c}_2=\frac{1}{2}, c_3=\frac{-39+\sqrt{241}}{40}, \tilde{c}_1=\frac{7-\sqrt{241}}{12}, \tilde{c}_3=\frac{-19+\sqrt{241}}{40},
\]
and choose the angles in such a way that all the sines are positive. Note that all the angles here are neither special nor exceptional. 
\end{example}

\begin{example}\label{ex:comm2} Consider a family of curvilinear two-gons $\mathcal{P}((\alpha_1, \alpha_2), (\ell, \ell))$ with sides of equal length. The characteristic equation becomes 
\[
\cos 2\sigma\ell =-\cos\left(\frac{\pi^2}{2\alpha_1}+ \frac{\pi^2}{2\alpha_2}\right),
\]
and we can therefore only recover  from its roots the side length $\ell$ and the quantity in the right-hand side. On the other hand, if we additionally restrict ourselves to  two-gons with $\alpha_1=\alpha_2$, then any two quasi-isospectral two-gons in this class are loosely equivalent. 
\end{example}

\begin{remark} The numerical computation of the first few eigenvalues in each of Examples \ref{ex:special1}--\ref{ex:comm2} indicates that the corresponding families or pairs of quasi-isospectral domains are not isospectral.
\end{remark}

\subsection*{Acknowledgements}\addcontentsline{toc}{section}{Acknowledgements} The authors are grateful to Pavel Kurasov for providing the paper \cite{KurSuhr}, as well as  to Chris Daw, Andrew Granville, and Peter Sarnak  for useful discussions. The research of LP was partially supported by EPSRC grants EP/J016829/1 and EP/P024793/1. The research of IP was partially supported by NSERC.

\end{document}